%% file: Paper.tex
\documentclass[11pt]{article}
\usepackage[margin=1.25in]{geometry}
\usepackage[utf8]{inputenc}
\input{pckgs}
\usepackage[backend=biber, style=alphabetic,]{biblatex}
\usepackage[framemethod=TikZ]{mdframed}
\title{Bayesian Inference with Shaped Deep Non-linear MLPs}
\author{
Boris Hanin\footnote{
Princeton University. Email: bhanin@princeton.edu.
}
,
Tianze Jiang\footnote{
Princeton University. Email: tzjiang@princeton.edu
}
}
\date{May, 2026}
\begin{document}
\maketitle
{
\begin{abstract}
A central aim of deep learning theory is to characterize how neural networks make predictions in the regime of simultaneously large model and training set size. Since the limits of diverging number of model parameters and dataset size do not commute it is not clear \textit{a priori} what limits exist. In this work, we shed new light on these questions by studying Bayesian inference in deep non-linear MLPs in the regime where the number of training samples ($P$), the input dimension ($N_0$), the hidden layer width ($N$), and the number of hidden layers ($L$) can all be large. We build on the Neural Covariance SDE \cite{li2022neural} to analyze predictive posteriors in the regime where $LP/N\in\Theta(1)$, playing the role of an effective network depth. Our framework covers both smooth and ReLU activation functions and applies to arbitrary temperature. We find to first order in $LP/N$ a simple criterion for which data generating processes benefit from depth in the sense that larger $LP/N$ increases the Bayesian model evidence. We also give a novel derivation of a prior result from the physics literature that at least to first order in $LP/N$, the Bayesian predictive posterior is remarkably simple and is simply equivalent to that of a data-dependent kernel method. 
    
\end{abstract}
}


\newpage
\section{Introduction}
From approximation to training dynamics and generalization, many important results in deep learning theory concern the fundamental setting of deep Multi-layer Perceptrons (MLPs) \cite{zhang2017understandingdeeplearningrequires, advani2017highdimensionaldynamicsgeneralizationerror, Belkin19reconciling, hanin2019universal, daubechies2022nonlinear, roberts2022principles}. A core aim of many of these articles is to rigorously characterize how scaling up different dimensions, such as the number of hidden layers ($L$), the width of each hidden layer ($N$), and the size of the dataset ($P$) jointly shape the behavior of large MLPs at initialization or during training. 

Many such works consider the limit of infinite width at fixed depth and dataset size, giving rise to by-now classical results concerning both the Neural Tangent Kernel (NTK) and the Neural Network Gaussian Process (NNGP) \cite{lee2017deep, chizat2019lazy,   du2019gradientdescentprovablyoptimizes, chizat2019lazy, jacot2020neuraltangentkernelconvergence, yang2021tensorprogramsiiineural} and mean-field limits \cite{chizat2019lazy, rotskoff2018neural, mei2018mean, sirignano2020mean}.
However, the infinite-width limit (especially in the NTK parameterization) suppresses important phenomena capturing the role of depth \cite{Hanin2019products, hayou2019training, hayou2022curse, li2022neural,roberts2022principles} when $L\propto N$ or  large dataset size $P\propto N$ \cite{Li_2021, Pacelli_2023,seroussi2023separation, naveh2021predicting, bassetti2024proportionalinfinitewidthinfinitedepthlimit}. More recent studies \cite{hanin2024bayesian, hanin24random, camilli2025informationtheoreticreductiondeepneural, li2025geometric} also focuses on proportional limits, in which depth, width, and data diverge jointly, revealing new stochastic and geometric structures.

Despite recent progress, our understanding of how jointly scaling architecture dimensions ($L, N, P$ and input dimension $N_0$) in non-linear networks impacts training-time fitting and test-time prediction remains incomplete. When and why do deep neural networks behave like kernel methods, and what enables feature-learning? When feature learning is possible, what features do MLPs learn? What combinations of data and architecture benefit from depth?

In this work, we present the first {mathematically rigorous} way to answer these questions on non-linear deep MLPs from the
Bayesian inference perspective \cite{Mac92, Nea94, hanin2023bayesian, hanin2024bayesian, bassetti2025featurelearningfinitewidthbayesian}. Bayesian inference analyzes a neural network architecture $x\mapsto f(x;\Theta
)$ by choosing a specific prior on the weights $\P(\Theta)$ and computing the posterior $\P(\Theta|\D)$ (conditional on the training dataset $\D$) via tilting the prior with the training loss $\P(\D|\Theta)\propto \exp(-\beta\L(\D; \Theta))$. The posterior network prediction on test data $x_0$ is therefore the pushforward of the weight posterior $\Theta\sim\P(\Theta|\D)$ to the space of functions outputs. The marginal of training data $\P(\D)$, equivalently the likelihood of the architecture and prior, can be used for model selection. 

Our main contribution in this paper is characterizing and justifying a rigorous unifying framework of analyzing Bayesian inference in non-linear MLPs when $LP\propto N$, building on the Neural Covariance SDE (NSDE) \cite{li2022neural, li2024differentialequationscalinglimits, noci2023shapedtransformerattentionmodels}. Our method of computing Bayesian statistics (outlined in \Cref{sec: main ideas} and elaborated in \Cref{sec: prior with nsde}) covers a wide range of non-linear activation functions (leaky ReLU in \Cref{sec: relu nsde} and smooth functions in \Cref{sec: smooth nsde}) and can naturally be applied at any temperature. As examples of applying our framework to answer questions about learning and inference, we compute the perturbative expansion at small effective-depth (\Cref{takeaway: perturb} and \Cref{sec: takeaways}). Our calculations in the perturbative expansion recover and generalize many results in prior literature for linear networks \cite{hanin2023bayesian} and weakly non-linear networks with cubic activation function \cite{hanin2024bayesian}. We further generalize these results and provide novel insights on the perturbative prediction feature map as well as the Bayesian evidence analysis with different activation functions.

\subsection{Preliminaries and technical setup}
\paragraph{Model.}
We consider MLPs with $L$ fully connected layers $f(x; \Theta): x\in\R^{N_{0}}\to y\in \R$\footnote{While our analysis easily extends to $y\in\R^d$ output for $d\in\Theta(1)$, we present scalar output results for simplicity.}:
\begin{equation}\label{eqn: forward pass}
    z_1=\frac{1}{\sqrt{N_{0}}} W_0 x, \quad \phi_{\ell}=\phi_s(z_{\ell}), \quad z_{\ell+1}=\sqrt{\frac{c}{N}} W_{\ell} \phi_{\ell}, \quad  y=z_{\mathrm{out}}=\sqrt{\frac{c}{N}} W_{\mathrm{out}} \phi_L\in\R
\end{equation}
where $\Theta = (W_0, W_1,\dots, W_{L-1}, W_{\operatorname{out}})$ is the set of weights, $z_i\in\R^N, i=1,2,\dots, L$ are hidden pre-activations, $c$ is a normalizing constant such that $
c^{-1}\triangleq \mathbb{E}_{g \sim \mathcal{N}(0,1)}\left[\phi_s(g)^2\right]
$ so each hidden neuron is $\N(0, 1)$ at init, $\phi_s$ is the shaped activation function with parameter $s$ (see below), and the interior recursion runs from $\ell=1,2,\dots, L-1$. At initialization (also as the network prior), we assume that all $W_i$'s are initialized with i.i.d. $\N(0, 1)$ entries. We wish to characterize the network posterior when $L, P, N\to\infty$ jointly under the squared loss
$
\mathcal{L}(\mathcal{D}; \Theta)\triangleq\frac{1}{2} \sum_{\mu=1}^P\left(f\left(x_\mu ; \Theta\right)-y_\mu\right)^2
$
over a training dataset
$
\mathcal{D}=\left\{\left(x_\mu, y_\mu\right), \mu=1, \ldots, P\right\}$, where $x_\mu \in \mathbb{R}^{N_0}, y_\mu \in \mathbb{R}$ are training data.

\paragraph{Proportional width and depth limit.} 
Our results concern a sequential limit in which we first  consider \emph{fixed finite} $P, N_0$ and take the $L, N\to\infty$ limit with $L\propto N$; then we take $P, N_0$ to be large. Prior work (e.g. \cite{hanin2023bayesian, hanin2024bayesian}) suggests these limits commute, but we do not establish this rigorously. For a fixed number of datapoints $P$, prior studies focusing on the product of $L$ random $N\times N$ matrices \cite{hanin2019universal, hanin2023bayesian, bassetti2024proportionalinfinitewidthinfinitedepthlimit,li2025geometric} indicate that $t = L/N$ controls the initialization behavior in deep linear networks when $L, N\to\infty$, and that both $t =0$ and $t =\infty$ yield degenerate scaling regimes. While such conclusion does not immediately extend to networks with a non-linear activation function, one can apply a \textit{shaping} mechanism to the non-linear activation function $\phi_s$ such that the ratio $t  = L/N$ still uniquely controls the initialization distribution in non-linear MLPs, as we see next.

\paragraph{Shaped activation function.} For a growing depth $L\to\infty$, to ensure that the output distribution at initialization is non-degenerate, various prior works \cite{martens2021rapid,li2022neural, zhang2022deep, hanin2024bayesian} have established the necessity of \emph{shaping} the activation function $\phi_s$. Specifically, for a smooth base non-linearity $\phi_{\operatorname{base}}(x)$ normalized to $\phi_{\operatorname{base}}(0)=\phi_{\operatorname{base}}'(0)-1=0$, we consider
\begin{equation}\label{eqn: shape smooth}
    \phi_s(x)=s\phi_{\operatorname{base}}(x/s) = x+\frac{1}{2s}\phi''_{\operatorname{base}}(0)\cdot x^2+\frac{1}{6s^2}\phi'''_{\operatorname{base}}(0)\cdot x^3+\dots
\end{equation}
for a shaping factor $s=s(N, L, P)$ that may depend on $N, L, P$. Intuitively, the shaping $s>1$ effectively diminishes the effect of non-linear activation at each layer ($\phi_s(x)\approx x$), such that the cumulative effect becomes controllable over a large number of $L\to\infty$ layers. Similarly, for leaky-ReLU $\phi_{\operatorname{base}}(x) = c_+\max(x, 0)+c_-\min(x, 0)$ we consider
\begin{equation}\label{eqn: shape relu}
    \phi_s(x)=x+\frac{1}{ {s}}\l({c_+}\max(x, 0)+c_-\min(x, 0)\r).
\end{equation}
We also refer to \cite{li2024differentialequationscalinglimits} for a rigorous discussion on why shaping is necessary and the different types of degenerate limits with mis-specified $s$. In short, there usually exists a critical $s^\star(N, L, P)$ such that shaping with $s\in\omega(s^\star)$ results in effective equivalence to a deep linear network ($\phi_s = \operatorname{id}$), and shaping with $s\in o(s^\star)$ results in a $L/N$-independent degenerate limit. Intuitively, the shaping mechanism ensures that the  (cumulative) effect of applying a non-linear activation function at each layer is on the same scale as the effect of linear layers.

\paragraph{Bayesian inference.}  The Bayesian inference \cite{Mac92, Nea94} point of view for studying \eqref{eqn: forward pass} is based on a prior distribution over the trainable weights and studying the posterior measure over the weights defined by:
$$\P_{\operatorname{post}}(\Theta|\D)\propto \P_{\operatorname{prior}}(\Theta)\exp(-\beta\mathcal{L}(\D; \Theta)),$$
where $\beta$ is the inverse temperature.
The characteristic function of the predictive posterior on a test data input $x_0\in\R^{N_0}$ is therefore:\footnote{While our computations seamlessly carry to characterizing the predictive posterior of multiple test data points jointly, we only derive the case of predicting one posterior for presentation here.}
\begin{equation}\label{eqn: characteristic function def}
    \mathbb{E}_{\operatorname{post}, \beta}[\exp [-i \kappa f(x_0 ; \Theta)]]=\frac{Z_\beta(x_0 , \kappa)}{Z_\beta(0)},
\end{equation}
where $Z_\beta$ is the \textit{partition function}:
\begin{equation}\label{eqn: bayes partition def}
    Z_\beta ({x}_{0}, {\kappa} )\triangleq{  (2 \pi \beta)^{P/2}}\cdot  \mathbb{E}_{\operatorname{prior}}\left[\exp \left[-\beta \L(\D; \Theta)-i\kappa f({x}_{0}; \Theta)\right]\right],\quad Z_{\beta}(0) = Z_{\beta}(x_0; 0).
\end{equation}
The marginal on the data, or equivalently the likelihood of the model and prior, is known as \textit{Bayesian model evidence} and takes the form
$$\P_{\operatorname{prior}}(\D)=  \mathbb{E}_{\operatorname{prior}}\left[\exp \left[-\beta \L(\D; \Theta) \right]\right]\propto Z_\beta({0}).$$
Maximizing evidence as an the objective is a canonical Bayesian approach to architecture/prior selection (see e.g. \cite{Mac92}).
\paragraph{Conjugate kernel under the prior.}
Because the partition function $Z$ in \eqref{eqn: bayes partition def} concerns only the prior distribution of $f$, it is sufficient to study the distribution of $f(\D\cup\{x_0\};\Theta)\in\R^{P+1}$ in \eqref{eqn: forward pass} at initialization.
An important observation of \eqref{eqn: forward pass} at initialization is that, for each $\ell$, conditioned on all prior weights $W_{[0:\ell-1]}$, the next layer $z_{\ell+1}\in\R^N$ is $N$ i.i.d. Gaussian random variables with mean zero. Over a finite set of inputs, the covariance matrix between individual neurons from different $z_{\ell+1}(x)$'s follows the \emph{Conjugate Kernel} $\Phi^{(\ell)}$ 
\cite{cho2009kernel, el2010spectrum,  pennington2017nonlinear, Fan20Spectra, li2022neural,benigni2022largesteigenvaluesconjugatekernel, chouard2023deterministic, wang2024nonlinearspikedcovariancematrices, } defined (on a pair of inputs 
$x_\alpha, x_\beta$) as:
\begin{equation}\label{eqn: conjugate kernel}
    \Phi^{(\ell)}_{\alpha\beta}\triangleq\cov\l(\l[z_{\ell+1}(x_\alpha)\r]_1,\l[z_{\ell+1}(x_\beta)\r]_1\r)= \frac{c}{N}\langle\phi_{\ell}(x_\alpha),\phi_{\ell}(x_\beta)\rangle.
\end{equation}
Furthermore, because weights at each layer are sampled independently, the following is a Markov chain for the forward pass:
$$X\to\Phi^{(0)}\to z_1\to \phi_1 \to \Phi^{(1)}\to z_2\to\phi_2\to \Phi^{(2)}\dots\to \Phi^{(L)}\to y $$
where the transition is only random at the $\Phi^{(\ell)}\to z_{\ell+1}$ steps.
In other words, it suffices to study the discrete Markov chain
\begin{equation}\label{eqn: Phi evo demo}
    \frac{1}{N_0}X^\top X=\Phi^{(0)}\to\Phi^{(1)}\to\dots\to\Phi^{(L)}
\end{equation}
with the final output of the model $ f({\D\cup\{x_0\}};\Theta)|W_{[0:L-1]}\sim\N(0, \Phi^{(L)})$. 
As a result, studying Bayesian inference reduces to studying the distributional properties (expectation over test functions) of conjugate kernel $\Phi^{(L)}$ under the prior with random hidden weights.

\section{Main idea: Bayesian inference via Neural Covariance SDEs}\label{sec: main ideas}
To formalize the connection between the conjugate kernel and Bayesian inference,
a crucial argument we use in this work (see \Cref{lem: Bayes partition}) is as follows: when $f(x) = W^\top h(x)$ for some output projection weight $W\in\R^{N\times 1}$ whose prior is i.i.d. Gaussian $\N(0, \I_N)$ and $\L$ is the squared loss, the partition function $Z_\beta\left({x}_{0}, {\kappa}\right)$ in \eqref{eqn: bayes partition def} has another equivalent form as an integral:
\begin{equation}\label{eqn: integral bayes partition}
    Z_\beta\left({x}_{0}, {\kappa}\right)= \int_{\R^P}\exp\l[-\frac{1}{2 \beta}\|p\|^2+i p^{\top}Y\r]\cdot \mathbb{E}_{\operatorname{prior}}\left[
    \exp\l(-\frac{1}{2}v^{\top}\Phi v\r)\right] \d p
\end{equation}
where $v = [p^{\top}, \kappa]^{\top}\in\R^{P+1}$ and $\Phi = \l[h(X), h(x_0)\r]^{\top}\l[h(X), h(x_0)\r]\in\R^{(P+1)\times (P+1)}$ is the matrix version of our conjugate kernel \eqref{eqn: conjugate kernel} at the penultimate layer. The formulation \eqref{eqn: integral bayes partition} also allows us to take $\lim_{\beta\to\infty} Z_\beta$ and is fairly standard \cite{seroussi2023separation,Li_2021,hanin2023bayesian}.
For other loss functions, one can also write the respective expressions of $Z$ with only $\E_{\operatorname{prior}}$ of statistics of $\Phi$. As an example, the Binary Cross-Entropy (BCE) loss gives equivalent partition function
$$
Z_\beta^{\mathrm{BCE}}\left(x_0, \kappa\right)=2^{-\beta P} \int_{\mathbb{R}_{+}^P} \nu_\beta^{\mathrm{PG}}(\!\!\d\omega) \int_{\mathbb{R}^P} \frac{\exp(-\frac{1}{2} p^{\top} \Omega^{-1} p)}{(2 \pi)^{P / 2} \sqrt{\operatorname{det} \Omega}} \cdot \mathbb{E}_{\operatorname{prior}}\left[
    \exp\l(-\frac{1}{2}v^{\top}\Phi v\r)\right] \d p
$$
where
$
\Omega=\operatorname{diag}\left(\omega_1, \ldots, \omega_P\right)$, $v=\left[
p+i \beta\left(Y-\frac{1}{2} \one\right) ,
\kappa\right]^\top,
$
and $\nu_\beta^{\mathrm{PG}}$ is the product Polya-Gamma law with marginals
$
\omega_\mu \sim \operatorname{PG}(\beta, 0)
$ \cite{polson2013bayesian}. In this paper, we focus specifically on the squared error because of the simplicity of \eqref{eqn: integral bayes partition} (specifically the integration form), but in principle our techniques to compute $\E_{\operatorname{prior}}[e^{-v^\top\Phi v/2}]$ provide solutions to cross-entropy and other loss functions as well. 

Having established that \eqref{eqn: integral bayes partition} only depends on the prior through $\mathbb{E}_{\operatorname{prior}}\left[
\exp\l(-\frac{1}{2}v^{\top}\Phi v\r)\right] $, it thus suffices to compute this expectation for each $p$. In the simplest case where the network prior to $\Phi$ is deterministically (or at least approximately) equal to a kernel $\bar\Phi$ with an associated \textit{feature map} $f: x_\alpha\to \bar x_\alpha$ such that $\bar\Phi_{\alpha\beta} = \langle \bar x_\alpha, \bar x_\beta\rangle$, the expectation in \eqref{eqn: integral bayes partition} is moot and Bayesian inference reduces to what we refer to as the \emph{kernel method} where the partition function $Z_\beta^{\operatorname{kernel}}\left({x}_{0}, {\kappa}; f\right)\triangleq\int_{\R^P}\exp\l[-\frac{1}{2 \beta}\|p\|^2+i p^{\top}Y-\frac{1}{2}v^{\top}\bar\Phi v\r] \d p$ can be computed directly:
\begin{align}\label{eqn: kernel method}
    \!Z_\beta^{\operatorname{kernel}}\!\!=&(2\pi)^{\frac{P}{2}}\!\exp \! \left[-\frac{1}{2} (Y^{\top} \!A^{-1} Y\!+\log\det A)-i \kappa \bar\Phi_{\kappa p} A^{-1} Y\!-\!\frac{1}{2} \kappa^2(\bar\Phi_{\kappa \kappa}\!-\!\bar\Phi_{\kappa p} A^{-1} \bar\Phi_{p \kappa})\right]
\end{align}
where $\bar\Phi=\l[\begin{array}{cc}
    \bar\Phi_{pp} &  \bar\Phi_{p\kappa}\\
     \bar\Phi_{\kappa p}& \bar\Phi_{\kappa\kappa} 
\end{array}\r]$ and $A = \frac{1}{\beta}\I_P+\bar\Phi_{pp}$.
This induces the predictive posterior \eqref{eqn: characteristic function def} to be an exact \emph{Gaussian} random variable with $
    \E_{\operatorname{post}}[f(x_0)] = \bar\Phi_{\kappa p}A^{-1}Y$ and $\var_{\operatorname{post}}(f(x_0)) = \bar\Phi_{\kappa \kappa}-\bar\Phi_{\kappa p} A^{-1} \bar\Phi_{p \kappa}$.
Intuitively, the predictive posterior mean weighs the training labels according to alignment between test and train in the feature space, and the variance measures the distance between test towards the orthogonal projection into train in the feature space.

In general, it is often not the case that the prior distribution of $\Phi$ concentrates on a single $\bar\Phi$. Without concentration, it is also not always true that the posterior prediction is a Gaussian (equivalently, $\log Z$ is not a quadratic polynomial of $\kappa$), hence the non-triviality of dealing with the expectation in \eqref{eqn: integral bayes partition}. 
For a finite $P$, we are interested in the proportional $L/N = t^\star>0$ limit, where there exist an important tool for analyzing the evolution \eqref{eqn: Phi evo demo}: the Neural Covariance SDE (NSDE) \cite{li2022neural, li2024differentialequationscalinglimits}, which (informally) states the following:
\begin{proposition}[Neural Covariance SDE, informally]\label{prop: nsde 1}
    Fix a terminal $t^\star>0$ and scale $N\to \infty, L = \lfloor t^\star N\rfloor$. For shaped smooth activations \eqref{eqn: shape smooth} and leaky ReLU \eqref{eqn: shape relu}, when taken $s=\alpha_{\operatorname{act}}^{-1/2}\sqrt{N}$ for some $\alpha_{\operatorname{act}}>0$, the interpolated covariance process $t\to \Phi^{({\lfloor t N\rfloor})}$ converges to a diffusion process $\Phi_t\in\R^{P\times P}$ solving\footnote{The exact form in \cite{li2022neural} is different but algebraically equivalent to ours, see \Cref{lem: NSDE matrix}.}
    \begin{equation}\label{eqn: nsde informal}
        \d \Phi_t= \alpha_{\operatorname{act}}\cdot b(\Phi_t) \d t+ \Phi_t^{1/2}\d B_t\Phi^{1/2},  t\in [0, t^\star],\quad \Phi_0 =\frac{1}{N_0}X^\top X.
    \end{equation}
    Here $B_t = \frac{1}{\sqrt{2}}\l(\tilde B_t+\tilde B^{\top}\r)\in\R^{P\times P}$ where $\tilde B_t\in\R^{P\times P}$ has i.i.d. Brownian motion entries. The drift components $b(\Phi)$ depends only on the base non-linearity $\phi_{\operatorname{base}}$.
\end{proposition} 
We formalize \Cref{prop: nsde 1} in \Cref{sec: prior with nsde}.
Recall that our goal is to evaluate the expectation under prior $\mathbb{E}\left[
    \exp\l(-\frac{1}{2}v^{\top}\Phi^{(L)} v\r)\right]$. To do that, simply
plugging in \eqref{eqn: nsde informal} for $s_t(v) \triangleq v^\top\Phi_t v\in\R$ yields
$$\d s_t(v) =v^{\top}\Phi_t^{1/2}\d B_t\Phi_t^{1/2} v +\alpha_{\operatorname{act}} v^\top b(\Phi_t)v\d t = \sqrt{{2}}s_t(v)\d W_t+\alpha_{\operatorname{act}} v^\top b(\Phi_t)v\d t$$
where $(W_t)_{t\geq 0}$ is a standard Brownian Motion on $\R$. The SDE on $s_t(v)$ thus (in theory) contains all the necessary information to solve $\E_{\operatorname{prior}}[\exp(-s_t(v)/2)]$, integrating which in \eqref{eqn: integral bayes partition} will then lead to a closed form solution to the Bayesian partition function.

Although our treatments for $P$ are non-asymptotic in nature, it is still convenient to balance scale preliminarily. Consider a deep linear network, in which \eqref{eqn: nsde informal} has $\alpha_{\operatorname{act}}=0$ or equivalently $s=\infty$, we have the linear SDE
$\d \Phi_t= \Phi_t^{1/2}\d B_t\Phi^{1/2},$
the RHS of which has operator norm (in fact the entire spectrum) $\l\|\Phi_t^{1/2}\d B_t\Phi^{1/2}\r\|_{op} \asymp \sqrt{P\d t}\|\Phi_t\|_{op}$. As a result, for large $P$, we will make a time change $\d\tau\triangleq P\d t$ such that \eqref{eqn: nsde informal} becomes (denote $c_{\operatorname{act}}\triangleq \alpha_{\operatorname{act}}P^{-1}$):
\begin{equation}\label{eqn: nsde tau}
    \d \Phi_\tau=c_{\operatorname{act}}\cdot b(\Phi_\tau) \d \tau+ \sqrt{\frac{1}{P}}\Phi_\tau^{1/2}\d B_\tau\Phi^{1/2}_\tau,\quad   \tau\in [0, \tau^\star \triangleq P t^\star=LP/N].
\end{equation}
and we have
\begin{equation}\label{eqn: s_tau}
    \d s_\tau(v)= \sqrt{{2}/P}\cdot  s_\tau(v)\d W_\tau+c_{\operatorname{act}}\cdot  v^\top b(\Phi_\tau)v\d \tau,\quad s_0(v) = v^\top \Phi_0 v.
\end{equation}
We will refer to $\tau\in [0, LP/N]$ as the \textit{depth-time}, which solely controls the $\Phi$ dynamics (normalized for large $P$), and $\tau^\star$ to be the terminal time. 
The resulting $ \frac{LP}{N}\in\Theta(1)$ regime has also been the central pursuit of characterizing the proportional $N, L, P\to\infty$ limits  in many related studies \cite{hanin2023bayesian, hanin2024bayesian, li2025geometric}.  Our results will be parameterized by a value of $\tau^\star>0$ and in the limit when we first take 
\begin{equation}\label{eqn: limit def}
\text{first~}N,L\rightarrow \infty\quad \text{with}\quad L/N\rightarrow \tau^*P^{-1}\qquad \text{then}\qquad P\to \infty.
\end{equation}
Methodologically, we first fix a terminal time $\tau^\star>0$. For any finite $P$, Bayesian statistics in the asymptotic limit with $L, N\to\infty$ proportionally such that $L/N =t^\star= \tau^\star P^{-1}$ can be computed as a function of $P$ and $\tau^\star$. Finally, we normalize the result so that the limit $P\to\infty$ is well-defined for fixed $\tau^\star$. Our calculation of the Bayesian partition function $Z^{(N, L, \D)}$ with \eqref{eqn: forward pass} will be based on the following result to $Z^{(\tau)}$ from solving the NSDE, which we justify in \Cref{sec: convergence apdx} (see also \Cref{sec: formality}).
\begin{theorem}[Pointwise convergence of MLP to SDE partition function]\label{thm: mlp part to nsde part}
Fix any $t^\star>0,   x_0\in\R^{N_0}, \kappa\in\R$ and finite $\beta>0$. In the width and depth limit $N, L\to\infty, L/N = t^\star$, $\tau^\star\triangleq Pt^\star$ with a fixed $P$, 
Bayesian inference with MLP \eqref{eqn: forward pass} from $\mathbb{E}_{\operatorname{post}, \beta}\l[e^{-i \kappa f(x_0 ; \Theta)}\r]\propto Z_\beta^{(N, L, \D)} ({x}_{0}, {\kappa})$ has that:
\begin{align}
    \lim_{N, L\to\infty} Z_\beta^{(N, L, \D)} ({x}_{0}, {\kappa})&= Z_\beta^{(\tau^\star)} ({x}_{0}, {\kappa} )\nonumber
    \\&= \int_{\R^P}\exp\l[-\frac{1}{2 \beta}\|p\|^2+i p^{\top}Y\r] \mathbb{E}_{\eqref{eqn: s_tau}}\left[ \exp\l(-\frac{1}{2}{
    s_{\tau^\star}(v)}\r)\right] \d p
    \label{eqn: nsde partition}
\end{align}
where $v = [p^\top, \kappa]^\top$ and $s_{Pt^\star}$ is the solution to the SDEs \eqref{eqn: nsde tau} and \eqref{eqn: s_tau}. 
\end{theorem}
From \Cref{thm: mlp part to nsde part}, one can also easily proceed to take $P\to\infty$ at \eqref{eqn: nsde partition}, which justifies dropping $O(P^{-1})$-terms at the RHS via the Lévy's continuity theorem. Due to the analytic form of $Z^{(\tau)}$, 
in the majority of the following results when not explicitly stated otherwise, we will analyze $Z=Z^{(\tau)}$ the {SDE} partition function, as opposed to the actual MLP partition function $\lim Z^{(N, L, \D)}$ (see \Cref{sec: formality}).

In general, while solving \eqref{eqn: s_tau} may be intractable for general $\tau^\star>0$, the small $\tau$ expansion (leading partial derivatives at $\tau=0$) of $\E[e^{-s/2}]$ is still tractable in closed form. Furthermore, the system can be explicitly closed in special cases, such as when $\alpha_{\operatorname{act}}=0$ (diffusion-only) and when $\tau^\star = c_{\operatorname{act}}^{-1}\to 0$ (equivalently drift-only), corresponding to different MLP scaling regimes. Our main technical contribution in this paper is thus to extract insights from partially solving \eqref{eqn: nsde tau} and \eqref{eqn: s_tau} and derive corollaries towards Bayesian inference in MLPs.

\paragraph{Assumptions.}  Our most important assumption is that $P<N_0$ and that the initial $\Phi_0=\frac{1}{N_0}X^\top X\in\R^{P\times P}\succ 0$ is full rank. Although our calculation of $Z^{(Pt)}$ does not require assumptions on the base activation function beyond those in the NSDE (see \Cref{sec: prior with nsde}), the final step of taking $P\to\infty$ and extracting qualitative results requires more assumptions in the dataset $\D$. 
Specifically, we assume that $\max_{\alpha}\|X_\alpha\|^2,  \min_{\alpha}\|X_\alpha\|^2, \|\Phi_0\|_{op}, \|\Phi_0^{-1}\|_{op}\in\Theta(1)$ which is true in Marchenko-Pastur-type Gram matrices (when $\lim P/N_0=\lambda<1$) and that $Y\in\R^P$ has $\Theta(1)$-entries.
When these large-$P$ assumptions are not met (e.g. if $\|\Phi_0\|_{op}\asymp P$), our finite $P$ partition function calculations are not affected, but the $P\to\infty$ limit will render different results (i.e. the set of dominating vs negligible terms may be different).

\section{In relation to prior literature}
\paragraph{Bayesian inference}
A central theme in modern studies of large networks is that neural networks at initialization can be studied in \emph{function space} rather than \emph{parameter space}. In the Bayesian view, randomness in the weights induces a prior over functions. 
Earlier work \cite{Nea94, Wil96} showed that for a broad class of priors, the single-hidden-layer network prior converges (as the width scales) to a Gaussian process, and that the resulting Gaussian process perspective can be leveraged for principled uncertainty quantification and inference. These ideas were later extended to {deep} fully connected networks by taking $L\to\infty$ after $N\to\infty$ first: iterating the layer-wise covariance map yields a deterministic kernel recursion at fixed depth, giving rise to the Neural Network Gaussian Process (NNGP) formalism \cite{lee2017deep, matthews2018gaussian, trevisan2023widedeepneuralnetworks}. In this regime, hidden feature Gram matrix (i.e. our conjugate kernel $\Phi$) concentrate in the large width limit, enabling a clean characterization of ``typical'' network behavior at initialization.

Outside of the strictly infinite-width setting, however, exact Bayesian inference is generally intractable, motivating a large body of approximation methods that aim to preserve calibrated uncertainty while remaining computationally feasible \cite{Blu+15,Gal+16}. Parallel to these algorithmic developments, recent theory has begun to characterize Bayesian posteriors in \textit{scaling limits} where computations become analyzable. For deep linear networks, \cite{hanin2023bayesian} first derives exact non-asymptotic expressions for the Bayesian evidence and predictive posterior and highlights the resulting \emph{effective posterior depth} $LP/N$. This is extended by \cite{bassetti2025featurelearningfinitewidthbayesian} which studies deep linear networks with convolution layers. In the resulting $LP/N\in\Theta(1)$ joint scaling limit with shaped nonlinear networks, \cite{hanin2024bayesian} develops perturbative first-order $LP/N$-expansions around the infinite-width $N>\!\!>LP$ baseline. The results in \cite{hanin2024bayesian} (many of which we will recover and extend, see \Cref{sec: technical results}) are closest in spirit to ours. However, their calculations are conducted at a physics level of rigor and rely on many simplifications of the data and the asymptotic limit. Our framework, on the other hand, is both mathematically rigorous, conceptually simple, and easier to generalize.

\paragraph{Infinite width and depth scaling}
Taking depth large introduces new phenomena even before considering training: stability of signal propagation, gradient explosion/vanishing, and emergent transitions between ordered and chaotic behavior. Early dynamical mean-field analyses of large random networks identified sharp transitions in typical dynamics and provided a conceptual template for later “edge-of-chaos” perspectives \cite{SCS88}. In deep \emph{feedforward} settings, mean-field methods quantify how correlations and norms propagate across layers, and connect expressivity to transient chaotic amplification mechanisms \cite{Poo+16}. Closely related work on “deep information propagation” formalizes how correlation maps and their fixed points govern trainability and motivates critical initialization schemes that keep signal propagation non-degenerate over many layers \cite{Sch+17, pennington2017nonlinear}. 

Complementing the mean-field approaches, various studies on product of many large random matrices at network initialization also give further insights to deep network initialization and conditioning \cite{hanin2019finite, Hanin2019products, hanin2021non,li2025geometric, hanin2025globaluniversalitysingularvalues}.
Within these, an important direction is to study the \emph{joint} limits in which width and depth grow together ($N,L\to\infty$ concurrently), a regime where deterministic infinite-width recursions are no longer adequate because finite-width fluctuations can accumulate over a large number of layers. For the forward pass of a single input token in the proportional $N\propto L$ limit, \cite{Hanin2019products} demonstrated that in the linear networks, accumulated feature fluctuations lead to log-normal distributions in hidden layers and gradients, revealing a rich stochastic structure absent in the deterministic kernel limit. \cite{li2022neural} furthers this discovery to the flexible framework of NSDEs that incorporates non-linear activation functions, which are later extended by \cite{ noci2023shapedtransformerattentionmodels, li2024differentialequationscalinglimits,} to ResNets. This joint-scaling perspective is complementary to fixed-depth infinite-width neural matrix law frameworks that characterize broad classes of architectures at initialization \cite{yang2021tensorprogramsiiineural}. In the present work, these simultaneous width and depth scalings are particularly natural because they yield stochastic limiting objects for Gram matrices \cite{li2022neural}, allowing spectral observables and Bayesian quantities to be expressed as functionals of the terminal covariance state. 

\section{Overview of technical takeaways}\label{sec: technical results}
In this section, we go into details of our exact technical results following expressing the partition function with \Cref{thm: mlp part to nsde part}. Each \texttt{Takeaway} is meant to be an informal statement which we expand in the appendix. They include both the recovery and extension of prior results derived with different methods as special cases of our framework, as well as novel results demonstrating our capability of solving previously open regimes. For example, while the positive temperature case $\beta<\infty$ was somewhat difficult to handle in the earlier frameworks of \cite{hanin2023bayesian, hanin2024bayesian}, a non-trivial $\beta$ fits easily into our calculations of \eqref{eqn: integral bayes partition} with \eqref{eqn: s_tau}. Finally, we provide insights extracted from Bayesian inference computations in different cases to answer questions regarding model architecture selection and feature learning.
\subsection{Recovering and extending prior results in Bayesian inference}\label{sec: recover results}
Our first set of results will come from special cases in \eqref{eqn: nsde tau}. Specifically, we consider when either the diffusion term dominates the SDE (equivalently $\alpha=0, s\to\infty$), which yields the deep linear network, or when the drift term dominates the SDE (equivalently $c_{\operatorname{act}}=\tau^{-1}\to \infty, \tau\to 0$), which yields the infinite-width $N>\!\!>LP$ limit.
\paragraph{Linear networks and diffusion-only} The simplest starting point in analyzing the network \eqref{eqn: forward pass} is when $\phi_s = \operatorname{id}$ the identity map. In this case $f$ is simply a deep linear network one only needs the limiting product of many large random matrices. Here, the Bayesian partition function $Z_{\beta}(x_0, \kappa)$ has a mathematically {closed} solution in the form of Meijer-G integrals \cite{hanin2023bayesian} at the $N, L, P\to\infty, LP/N \in\Theta(1)$ limit. Below in \texttt{Takeaway 1.1}, we recover two main results in \cite{hanin2023bayesian}.
\begin{itemize}
\item 
\texttt{Takeaway 1.1}: \textbf{Deep linear network}: The Bayesian partition function in the linear network $Z^{\operatorname{linear}}_{\beta}(x_0, \kappa)$ has the explicit integration form (for any $\tau$):
$$Z_\beta^{\operatorname{linear}(\tau )}\propto\int_{\R^P}\mathbb{E}\left[
    \exp\l(-\frac{1}{2 \beta}\|p\|^2+i p^{\top}Y-\frac{1}{2}\exp \left(-\frac{\tau }{P}+\sqrt{\frac{2 \tau }{P}} G\right)v^{\top}\Phi_0 v\r)\right] \d p$$
where expectation is over $G\sim\N(0, 1)$ and $v = [p^\top, \kappa]^\top$. As an exemplary corollary,
the first-order expansion at $\tau=0$ with $\beta=\infty$ is:
$$\left.\frac{\partial\log  Z^{\operatorname{linear}(\tau)}_{\infty}(\cdot )}{\partial \tau}\right|_{\tau=0}=\frac{1+O(P^{-1})}{4P}\left[\left(P(1-\nu_0)+{\kappa^2}\|x_0^\perp\|^2\right)^2+2P(1-2\nu_0)\right]$$
where $\nu_0 = P^{-1}Y^\top\Phi_0^{-1}Y$ is the linear interpolator norm of $\D$, and $x_0^\perp\perp\operatorname{span}(X_{\D})$ is the orthogonal component projecting test  $x_0$ to the span of the training inputs $X_{\D}$.
\end{itemize}
This matches at large $P$ the expansion from \cite{hanin2023bayesian}. We will show \texttt{Takeaway 1.1} with our NSDE framework in \Cref{sec: special 1 linear}.
\paragraph{Infinite-width kernel limit.}
While non-linear networks are remarkably more complicated than their linear counterparts, the infinite-width $\tau_{\operatorname{terminal}} =LP/N \to 0$ limit, which can be shown to effectively kills the diffusion terms by a time-change and leaves an equivalent \textbf{Neural Covariance ODE}, is solvable. In this limit, without the activation function, the conjugate kernel of linear networks is trivially $\Phi_L=\Phi_0$ with the identity feature map. This trivial reduction is also true if $s\in\Theta(N^{ 1/2})$ by \Cref{prop: nsde 1}. Alternatively, we show that a stronger shaping $s\in\Theta(L^{ 1/2})$ (which is equivalent to $s\in\Theta(N^{ 1/2})$ with positive $\lim L/N>0$)
induces a non-trivial deterministic kernel method when $\lim L/N= 0$. This fact was established in \cite{hanin2024bayesian} with cubic activation $\phi_{\operatorname{base}} = x+\psi x^3, \phi_s(x) = x+\frac{\psi}{L}x^3$, which we will extend to general activation functions. Since an ODE guaranties a deterministic answer, Bayesian inference reduces to a kernel method, which gives a deterministic yet training-data agnostic feature map (that explicitly only depends on $c, \phi_{\operatorname{base}}$). 
\begin{itemize}
\item 
\texttt{Takeaway 1.2}: \textbf{Data agnostic feature map in the infinite width limit}: In the large width $N>\!\!>LP$ limit, Bayesian inference (at any temperature) for MLP \eqref{eqn: forward pass} with shaping \eqref{eqn: shape relu} and \eqref{eqn: shape smooth} where $s=(c_{\operatorname{act}}/L)^{-1/2}$, reduces to a data-agnostic kernel method \eqref{eqn: kernel method} with a feature map $f_{0}(x; c_{\operatorname{act}}, \phi_{\operatorname{base}})\in\R^{\infty}$.
In the zero temperature case, the predictive posterior at test $x_0$ is a Gaussian with mean and variance
    \begin{equation}\label{eqn: kernel mean and var}
        \mu_0(x_0) = \sum_{\mu\in\D}a_\mu Y_\mu,\quad \sigma_0^2(x_0) = \|f_0(x_0)^{\perp}\|^2
    \end{equation}
    where $f_0(x_0)=f_0(x_0)^{\perp}+\sum_{\mu\in\D}a_\mu f_0(x_\mu)$ and $f_0(x_0)^{\perp}$ is orthogonal to $\operatorname{span}_\mu(f_0(x_\mu))$.
\end{itemize}
We justify this statement precisely in \Cref{sec: special 2 neural ode}.
It will also be apparent that $c_{\operatorname{act}}\in o(1)$ implies a trivial feature map (equivalent to the linear network), and $c_{\operatorname{act}}=\infty$ admits a unique (although possibly divergent) limit for the kernel, thus the criticality  $s=\Theta(L^{1/2})$ at the infinite-width limit.
Since we mostly assume a fixed $\phi_{\operatorname{base}}$, we will also use $f_0(x; c_{\operatorname{act}})$ to short-hand the feature map and discuss the role of $c_{\operatorname{act}}$ later.

\subsection{Small depth-time perturbation}\label{takeaway: perturb}
Now we turn  to the more general case where neither the infinite width non-linearity induced kernel nor the linear network effect dominates the partition function. 
Firstly, we need to decide the shaping factor of $\phi_s$, such that non-linearity effects contribute at the same order as linear networks. Motivated by the infinite-width limit as well as \eqref{eqn: nsde tau}, we consider the following shaping:
\begin{equation}\label{eqn: shape def}
        s_{\operatorname{shape}} = \sqrt{c_{\operatorname{act}}^{-1}{N}/{P} }=\sqrt{c_{\operatorname{act}}^{-1}{L}/{\tau}}
\end{equation}
with $c_{\operatorname{act}}\in\Theta(1)$. 
Our main technical result concerns perturbatively solving \eqref{eqn: bayes partition def} under \eqref{eqn: shape def} in the first order of $\tau = LP/N$. Noticing that $Z^{(0)}$ is the kernel method \eqref{eqn: kernel method} with input kernel $\bar\Phi = \Phi_0$ regardless of model, we can thus compare $\frac{1}{Z^{(0)}}\l.\partial_\tau Z\r|_{\tau = 0} = \l.\partial_\tau \log Z\r|_{\tau = 0}$.
\begin{itemize}
    \item \texttt{Takeaway 2}: \textbf{Perturbative expansion of $Z^{(\tau)}$.} 
    The 1st order partial derivative of small $\tau = \frac{LP}{N}$ in the joint limit \eqref{eqn: limit def} with shaping according to \eqref{eqn: shape def}:
    \begin{equation}\label{eqn: perturbative meta}
        \left. {\partial_\tau \log Z_{\beta}^{(\tau)}(\cdot )}\right|_{\tau=0}\!\!=
        \left.\partial_\tau \log Z^{\operatorname{linear}}_{\beta}(\cdot ) \right|_{\tau=0}\!\!+c_{\operatorname{act}} \left. {\partial_c\log  Z^{\operatorname{kernel}}_{\beta}(\cdot ; f_0(\cdot; c))} \right|_{c=0}
    \end{equation}
    is the \textit{sum} of first order derivative in the
    \textit{linear network component} \text{in \texttt{Takeaway 1.1}} and 
    \textit{infinite-width kernel component} with partition function \eqref{eqn: kernel method}  \text{in \texttt{Takeaway 1.2}}.
\end{itemize}
Balancing the two terms in \texttt{Takeaway 2} immediately implies the critical $c_{\operatorname{act}}\in\Theta(1)$ in \eqref{eqn: shape def} is indeed the appropriate scaling. We show \eqref{eqn: perturbative meta}, as well as the following set of corollary takeaways in \Cref{sec: small depth-time expansion}. 


\paragraph{Perturbative Bayesian evidence} Eqn. \eqref{eqn: perturbative meta} is already a strong statement that allows us to draw interesting corollaries. For instance, this implies that first order Bayes {evidence} $\P_{\operatorname{prior}}(\D) \propto Z_{\beta}(0)$ in \eqref{eqn: bayes partition def} is just the sum of the linear network evidence and the infinite-width kernel evidence. See \Cref{sec: interpretations} for an elaborate account.
\begin{itemize}
    \item \texttt{Takeaway 2.1}: \textbf{Bayesian evidence at small $\tau $.} 
    In the first order of small $\tau$ expansion in the joint limit \eqref{eqn: limit def}, the Bayesian evidence is given by (at $\tau=0$ and up to $O(P^{-1})$ remaining terms):
    \begin{equation}\label{eqn: perturbative evidence}
 \frac{\partial\log  Z^{(\tau)}_{\beta}(0)}{P\partial \tau}\!\! =\frac{1}{4}\l(\nu_0^\beta-\frac{1}{P}\operatorname{Tr}\left(\Phi_0(\Phi_0+\beta^{-1} \I)^{-1}\right)\r)^2+ \frac{c_{\operatorname{act}}}{P}\frac{\partial \log Z^{\operatorname{kernel}}_{\beta}(0; f_0)}{\partial c}
    \end{equation}
    where $\nu_0^\beta = P^{-1}Y^\top(\Phi_0+\beta^{-1} \I)^{-1}\Phi_0(\Phi_0+\beta^{-1} \I)^{-1}Y$ is the (normalized) RKHS norm of the Kernel Ridge Regression estimator. We normalize both sides by $1/P$ because the log marginal of data is $\log\P(\D)\propto P$. When $c_{\operatorname{act}}=0$, the remaining first term is equivalent to that of the linear network in \texttt{Takeaway 1.1}.
    \end{itemize}
It is notable that while the linear network evidence can be shown to always increase with depth, such is not true with the infinite-width kernel in \texttt{Takeaway 1.2}. As a result, depending on the exact scale of $c_{\operatorname{act}}$, we may encounter phase transitions in the form of ``depth is only beneficial if non-linearity is not too strong''. We list some specific corollaries of this flavor with concrete $\D$ in \Cref{sec: takeaways}.

\paragraph{Adaptive prediction kernel at $\tau>0$}

Another corollary of \eqref{eqn: perturbative meta} is that we can derive an equivalent feature map for the predictive posterior in the zero temperature regime. We show this in \Cref{sec: first-order predictive posterior}.
\begin{itemize}
    \item \texttt{Takeaway 2.2}: \textbf{Equivalent data-dependent feature map at small $\tau$.} 
    In the first order of small $\tau$ expansion in the joint limit \eqref{eqn: limit def} at $\beta=\infty$, the Bayesian predictive posterior of at test data $x_0$ is (approximately) equivalently given by kernel method with the training-data dependent feature map:
    \begin{equation}\label{eqn: first order feature map}
        x_\alpha\to \sum_{x_\mu\in\D} a_\mu f_{0}(x_\mu; c_{\operatorname{act}}\tau)+\lrp{1+\frac{1}{2}(\nu_0-1)\tau} f^{\perp}_{0}(x_\alpha; c_{\operatorname{act}}\tau)
    \end{equation}
    where $a_\mu$ follows the definition in \eqref{eqn: kernel mean and var}. The predictive mean remains un-changed from \eqref{eqn: kernel mean and var}, whereas the variance is reshaped as a function of the training dataset.
    \end{itemize}
By approximately given, we mean that we drop $O(P^{-1})+O(\tau^2)$ terms in the partition function $Z$.
An important remark is that while the predictive posterior is equivalently given by a kernel method with an explicit feature map, the full Bayesian partition function is \emph{not} given by the same kernel method (in contrast to \texttt{Takeaway 1.2}). Such equivalence is only possible by the fact that the leading order derivative of the partition function $\partial_\tau\E_{\operatorname{post}}[e^{-i\kappa f(x_0)}] = \partial_\tau \l[\log Z(x_0, \kappa)-\log Z(x_0, 0)\r]$ is a \textit{quadratic} polynomial of $\kappa$ (after dropping $o(1)$ terms) at  $\tau=0$ (\Cref{prop: first order predictive posterior gaussianity}).

\subsection{Interpretations of perturbative results}\label{sec: takeaways}
Given the above results, it is natural to ask: what do our technical takeaways imply in terms of model selection and prediction? In particular, what data distribution does deeper network or non-linearity prefer or penalize, as evaluated by the Bayesian evidence $\partial_\tau\P(\D)\geq 0$?

We make some concrete interpretations of our findings when $\Phi_0 = \rho\I + o(1)$, i.e., when the input data have a fixed norm $\rho>0$ and are (approximately) orthogonal to each other. Such conditions typically arise  when input data are from isotropic Gaussians, or when batch norm is applied.
In the proportional joint limit with a small terminal $LP/N$, we justify two novel findings from our Bayes calculations in the form of when data favors depth or non-linear activation strength in the model likelihood. We show these results in \Cref{sec: interpretations}.

\begin{itemize}
    \item \texttt{Takeaway 3.1:} \textbf{Smooth activation.} Under normalized inputs, for deep MLPs with $c_{\operatorname{act}}$-shaped smooth activation \eqref{eqn: shape smooth} and \eqref{eqn: shape def}, depth is beneficial (at first order $LP/N$) to fitting if and only if:
    \begin{align*}
        0\leq &\frac{1}{4}\left(\frac{\rho\|y\|^2}{\left(\rho+\beta^{-1}\right)^2}-\frac{\rho}{\rho+\beta^{-1}}\right)^2 +c_{\operatorname{act}} \left\{\frac{\rho\left(2\left(c_1+c_2\right) \rho-\left(3 c_1+2 c_2\right)\right)}{2 (\rho+\beta^{-1} )^2}\|y\|^2\right.
        \\&\quad +\left.\frac{c_1 \rho^2}{2 (\rho+\beta^{-1} )^2}\|\tilde{y}\|^2-\frac{\left(3 c_1+2 c_2\right) \rho(\rho-1)}{2 (\rho+\beta^{-1} )}\right\}
    \end{align*}
    where $c_1 =\frac{1}{4} (\phi_{\operatorname{base}}''(0))^2, c_2 = \frac{1}{2} \phi_{\operatorname{base}}'''(0)$, and $\|y\|^2 = \frac{1}{P}\sum_\D Y_\alpha^2, \|\tilde y\|^2 = \frac{1}{P}(\sum_\D Y_\alpha)^2$.
    \item \texttt{Takeaway 3.2:} \textbf{ReLU activation.}  For deep MLPs with $c_{\operatorname{act}}$-shaped ReLU activation \eqref{eqn: shape relu} and \eqref{eqn: shape def}, depth is beneficial  if and only if:
    $$0\leq \frac{1}{4}\left(\frac{\rho\|y\|^2}{\left(\rho+\beta^{-1}\right)^2}-\frac{\rho}{\rho+\beta^{-1}}\right)^2+c_{\operatorname{act}}\cdot \frac{\left(c_{+}-c_{-}\right)^2\rho}{4 \pi(\rho+\beta^{-1})^2}\left( \|\tilde y\|^2- \|y\|^2\right)$$
    where $\|y\|^2 = \frac{1}{P}\sum_\D Y_\alpha^2, \|\tilde y\|^2 = \frac{1}{P}(\sum_\D Y_\alpha)^2$.
\end{itemize}
\subsection{Perturbative convergence from MLP to SDE}\label{sec: formality}

Because our above takeaways use the SDE partition functions, not statistics from the actual MLP,
it is thus necessary to explain how they are mathematically rigorous and what types of limits are necessary to translate SDE takeaways to MLPs. \Cref{thm: mlp part to nsde part} gives us point-wise convergence (for each $\tau$) along the forward passage, but perturbative takeaways do not immediately follow. While taking the perturbative limit in discrete Markov chains is not well-defined, we can show a stronger convergence from MLP layer-wise to SDE, which we will justify (alongside with the proof of \Cref{thm: mlp part to nsde part}) in \Cref{sec: convergence apdx}.

\begin{itemize}
    \item \texttt{Takeaway 4:} \textbf{Uniform convergence of generator.} Fix $v\in\R^{P+1}$ and time $ T>0$ and let $F(\Phi)\triangleq e^{-\frac{1}{2} v^{\top} \Phi v}$. Define the rescaled discrete-time difference per layer in \eqref{eqn: forward pass}
$$
D_N F(t, \Phi_0)\triangleq N\left(\mathbb{E}_{\operatorname{prior}}\left[F\left(\Phi^{(\lfloor t N\rfloor+1)}\right)\right]-\mathbb{E}_{\operatorname{prior}}\left[F\left(\Phi^{(\lfloor t N\rfloor)}\right)\right]\right),
$$
and the limit (taken from applying Ito's Lemma on $F(\Phi_t)$)
$$
D F\left(t, \Phi_0\right)\triangleq\mathbb{E}_{\eqref{eqn: nsde informal}}\left[\mathcal{L} F\left(\Phi_t\right)\right]=\mathbb{E}_{\eqref{eqn: nsde informal}}\left[F\left(\Phi_t\right)\left(-\frac{1}{2}\alpha_{\operatorname{act}} v^{\top} b\left(\Phi_t\right) v+\frac{1}{4}\left(v^{\top} \Phi_t v\right)^2\right)\right].
$$
Then for any compact $K$,
$$
\sup _{t \in[0, T], \Phi_0 \in K}\left|D_N F\left(t, \Phi_0\right)-D F\left(t, \Phi_0\right)\right| \to  0.
$$
\end{itemize}
As a corollary with dominated convergence theorem, for any $x_0, \kappa$ and $\ell, \beta>0$, we have
\begin{equation}\label{eqn: perturb Z}
     \frac{N}{P}\l(Z_\beta^{(N, \ell+1, \D)}-  Z_\beta^{(N, \ell, \D)}\r)\to  \l.\partial_\tau Z_\beta^{(\tau)} ({x}_{0}, {\kappa} )\r|_{\tau = P\ell/N}.
\end{equation} We present this convergence formally in \Cref{prop: unif convergence}.

\printbibliography
\appendix

\newpage
\section{Computing the prior: Neural Covariance SDE} \label{sec: prior with nsde}
Let us first lay out the foundations of NSDE as we need.
Recall again the forward pass \eqref{eqn: forward pass}:
$$    z_1=\frac{1}{\sqrt{N_{0}}} W_0 x, \quad \phi_{\ell}=\phi_s(z_{\ell}), \quad z_{\ell+1}=\sqrt{\frac{c}{N}} W_{\ell} \phi_{\ell}, \quad  y=z_{\mathrm{out}}=\sqrt{\frac{c}{N}} W_{\mathrm{out}} \phi_L\in\R$$
Consider network weights at initialization with shaped activation $\phi_s$ in \eqref{eqn: forward pass} on a finite dataset $\D$.
The key is to analyze how $\Phi^{(\ell)}=\l[\frac{c}{N}\langle\phi_{\ell}^\alpha, \phi_{\ell}^\beta \rangle\r]_{\alpha, \beta\in [P]}$ changes from layer $\ell$ to $\ell+1$. Using the conditional Gaussian structure of $\left(g_{\ell}^\alpha\right)_\alpha$ and a Taylor expansion of $\phi_s$ at zero, \cite{li2022neural} shows that for large $n$ the update:
$$
\Phi^{(\ell+1)}=\Phi^{(\ell)}+\frac{1}{N} b(\Phi^{(\ell)})+\frac{1}{\sqrt{N}} (\Phi^{(\ell)})^{1 / 2} \xi_{\ell}(\Phi^{(\ell)})^{1 / 2}  +\text {higher ordered terms}
$$
where $\xi_{\ell}\in\R^{P\times P}$ is sampled from a Gaussian Orthogonal Ensemble , and $b$ is a deterministic bias functions that connects with the activation. Intuitively, the $1 / \sqrt{N}$ zero-mean fluctuations come from the variance averaging over $N$ neurons; the drift $b(\Phi)$ comes from the first non-vanishing nonlinear terms in the Taylor series of the activation function.
We then let $N \to  \infty$ and $L \to  \infty$ with
$
\frac{L}{N} \to  T \in(0, \infty)
$
and define the continuous ``depth-time'' index $t \in[0, T]$ via $\ell \approx t N$. Interpolating $\Phi^{(\ell)}$ as a càdlàg process $\Phi^{(\lfloor t N\rfloor)}\triangleq \Phi_t^{(N)}$, Ethier-Kurtz-type results on convergence of Markov chains to diffusions imply our convergence to the solution of a stochastic differential equation, as we see below.
\subsection{Local convergence}
Let us first set up the proper notion of local convergence needed from the MLP to the Neural Covariance SDE.

\begin{definition}[Local convergence in the Skorokhod topology]
We say a sequence of processes $X^n$ converge locally to $X$ in the Skorokhod topology under a continuous test function $f: \cdot \to [0, \infty]$ if for any $r>0$, the following stopping times
$$
\tau^{(n)}_r\triangleq \inf \left\{t \geq 0:f(X_t^n) \geq r\right\}, \quad \tau_r\triangleq \inf\left\{t \geq 0: f(X_t) \geq r\right\}
$$
has that $X_{t \wedge \tau^{(n)}}^n\Rightarrow X_{t \wedge \tau}$ in the Skorokhod topology.
Furthermore, we say that the process $X_t$ does not have finite-time explosion if $\P(\lim_{r\to\infty}\tau_r=\infty)=1$.
\end{definition}
Building on the definition, we can immediately show the following (proof deferred to \Cref{sec: lemmas}).
\begin{lemma}[Non-explosion of pre-limit]\label{lem: finite non-explosion}
    Suppose a sequence of processes $X^n$ converge locally to $X$ in the Skorokhod topology under a test function $f: \cdot \to [0, \infty]$, and that the limit process $X_t$ has continuous paths and does not have finite-time explosion. Then for any $r, t> 0$:
    $$\limsup_n\P\l(\tau_r^{(n)}\leq t\r)\leq \P(\tau_r\le t)$$
    and as a result $\lim_{r\to\infty}\limsup_n\P\l(\tau_r^{(n)}\leq t\r)=0$.
\end{lemma}

The main reason for the necessity of local convergence is that the drift terms in \Cref{eqn: nsde informal} are not always globally Lipschitz. To that end, we had to introduce the stopping times $\tau_r$ such that under $F_r$, the drift is bounded and Lipschitz. Let us now specify the exact statement for convergence from MLP to Neural Covariance SDE below.
\subsection{ReLU networks} \label{sec: relu nsde}
Consider the network defined in \eqref{eqn: forward pass} where $\phi_s(x) = s_{+} \max (x, 0)+s_{-} \min (x, 0)$ is (leaky) ReLU-like with shaping from \eqref{eqn: shape relu}. The NSDE is formally defined as follows.
\begin{proposition}[NSDE in ReLU networks, see Theorem 3.2 in \cite{li2022neural}]\label{thm: nsde relu}
In the notations of \eqref{eqn: forward pass},
let $\Phi^{(\ell)}\in\R^{P\times P}$ where $\Phi^{(\ell)}_{\alpha \beta}\triangleq \frac{c}{N}\left\langle\phi_{\ell}^\alpha, \phi_{\ell}^\beta\right\rangle$ and shaped activation
$$\phi_s(x)=x+\frac{1}{s}\l( c_{+}\max (x, 0)+ {c_{-}} \min (x, 0)\r),\quad s=(N/\alpha_{\operatorname{act}})^{1/2} .$$
In the limit $N \to \infty, \frac{L}{N} \to T<\infty$, the interpolated process $\Phi^{(\lfloor t n\rfloor)}$ converges locally in distribution in the Skorokhod topology of $D_{\mathbb{R}_{+}, \mathbb{R}^{P\times P}}$ under the test function $f(\Phi) = \max\{(\min_\alpha \Phi_{\alpha\alpha})^{-1}, \max_\alpha\Phi_{\alpha\alpha}\}$ without finite-time explosion to the solution of the SDE
$$
\d \Phi_t=\alpha_{\operatorname{act}}\cdot b(\Phi_t) \d t+\Phi_t^{1/2}\d B_t\Phi_t^{1/2}, \quad \Phi_0=\left[\frac{1}{N_{0}}\left\langle x^\alpha, x^\beta\right\rangle\right]_{1 \leq \alpha \leq \beta \leq P},
$$
where if we let $\rho(x)\triangleq\frac{\left(c_{+}-c_{-}\right)^2}{2 \pi}\left(\sqrt{1-x^2}-x\arccos x\right)$, $D_t \triangleq \operatorname{diag}(\Phi_t)$, then
$$
b(\Phi_t)=D_t^{1/2}\rho\l(D_t^{-1/2}\Phi_tD_t^{-1/2}\r)D_t^{1/2} \in\R^{P\times P}
$$
in which $\rho$ is applied entry-wise to the normalized $P\times P$ matrix.   
\end{proposition}
In particular, $\rho(1)=0$. As a result, the diagonal entries follow a geometric Brownian Motion without drift, which explains the finite-time non-explosion.
We will also use the fact that the drift-only ODE evolution (\Cref{thm: nsde relu} dropping the diffusion term) admits a deterministic feature map independent of $\D$.
\begin{proposition}[ODE feature map, ReLU]\label{prop: relu ode feature map}
    There exists a feature map $f: x\in\R^{N_0}, c\in\R  \to\R^{\infty}$ such that the $P\times P$ kernel matrix
    $$(\Phi_c)_{\alpha\beta}=\langle f( x_\alpha, c), f(x_\beta, c) \rangle$$
    satisfies the matrix ODE (following notations of \Cref{thm: nsde relu}),
    $$\d \Phi_c = b(\Phi_c) \d c,\quad b(\Phi_t)=D_t^{1/2}\rho\l(D_t^{-1/2}\Phi_tD_t^{-1/2}\r)D_t^{1/2} \in\R^{P\times P}$$
\end{proposition}
\begin{proof}[Proof of \Cref{prop: relu ode feature map}]
    From \Cref{thm: neural cov ode} and \Cref{lem: psd of limit}, we know that the ODE path is always PSD. For any finite set $S \subset \mathbb{R}^{N_0}$, solve the ODE for the restriction $\Phi_c^S$. Since the vector field is entry-wise and compatible with restriction to subsets, these finite-dimensional solutions are consistent under restriction. Hence they define a positive-definite kernel $K_c\left(x, x^{\prime}\right)$ on the whole input space, and Moore-Aronszajn applies. Thus, a feature representation induced by the respective RKHS exists. Because this kernel operates on the entire space $\R^{N_0}$ of possible inputs, it cannot be dependent on the specific $\D$.
\end{proof}
\subsection{Smooth networks} 
\label{sec: smooth nsde}
Consider smooth base activation function $\phi_{\operatorname{base}}\in C^4(\R)$ before shaping via \eqref{eqn: shape smooth} that satisfies: 
\begin{enumerate}
    \item[\texttt{A1}] First two orders of derivatives $0=\phi(0)=\phi'(0)-1$.
    \item[\texttt{A2}] There exists some $C>0$ such that $|\phi^{(4)}(x)| \leq C\left(1+|x|^{C}\right)$.
    \item[\texttt{A3}] Negative third order derivative $\frac{3}{4} \phi^{\prime \prime}(0)^2+\phi^{\prime \prime \prime}(0) \leq 0$.
\end{enumerate}
The second and third assumptions are technical ones arising from the derivation of NSDE, we refer to \cite{li2022neural} for a careful narrative of their necessity. These assumptions were satisfied by common activation functions such as Sigmoid, Tanh, and Soft-Plus.
\begin{proposition}[See Theorem 3.9 and Proposition 3.7 in \cite{li2022neural}]\label{thm: nsde smooth}
    Assuming that $\phi$ satisfies assumptions \emph{\texttt{A1-3}}. Consider the forward pass \eqref{eqn: forward pass} with shaping $\phi_s(x)=s \phi(x / s), s= (\alpha_{\operatorname{act}}/N)^{-1/2}$.
In the limit $N \to \infty, \frac{L}{N} \to T<\infty$, the interpolated process $\Phi^{(\lfloor t n\rfloor)}$ converges locally in distribution in the Skorokhod topology of $D_{\mathbb{R}_{+}, \mathbb{R}^{P\times P}}$ under the $f(\Phi) = \max\{(\min_\alpha \Phi_{\alpha\alpha})^{-1}, \max_\alpha\Phi_{\alpha\alpha}\}$ without finite-time explosion to the solution of the SDE
$$
\d \Phi_t=\alpha_{\operatorname{act}}\cdot b(\Phi_t) \d t+\Phi_t^{1/2}\d B_t\Phi_t^{1/2}, \quad \Phi_0=\left[\frac{1}{N_{0}}\left\langle x^\alpha, x^\beta\right\rangle\right]_{1 \leq \alpha \leq \beta \leq P},
$$
Here $B_t = \frac{1}{\sqrt{2}}\l(\tilde B_t+\tilde B^{\top}\r)\in\R^{P\times P}$ where $\tilde B_t\in\R^{P\times P}$ has i.i.d. Brownian motion entries. The drift components $b(\Phi)$ is given by:
$$
b_{\alpha \beta}(\Phi)=\frac{1}{4}\phi^{\prime \prime}(0)^2\left[\Phi_{\alpha \alpha} \Phi_{\beta \beta}+\Phi_{\alpha \beta}\left(2 \Phi_{\alpha \beta}-3\right)\right]+\frac{1}{2}\phi^{\prime \prime \prime}(0) \Phi_{\alpha \beta}\left(\Phi_{\alpha \alpha}+\Phi_{\beta \beta}-2\right).
$$
for all $\alpha,\beta \in [P]$.
\end{proposition}
Following \Cref{prop: relu ode feature map}, we also have the ODE feature map statement.
\begin{proposition}[ODE feature map, Smooth]\label{prop: smooth ode feature map}
     There exists a feature map $f: x\in\R^{N_0}, c\in\R  \to\R^{\infty}$ such that the $P\times P$ kernel matrix
    $(\Phi_c)_{\alpha\beta}=\langle f( x_\alpha, c), f(x_\beta, c) \rangle$
    satisfies the matrix ODE $\d \Phi_c = b(\Phi_c) \d c$ where
    $$b_{\alpha \beta}(\Phi)=\frac{1}{4}\phi^{\prime \prime}(0)^2\left[\Phi_{\alpha \alpha} \Phi_{\beta \beta}+\Phi_{\alpha \beta}\left(2 \Phi_{\alpha \beta}-3\right)\right]+\frac{1}{2}\phi^{\prime \prime \prime}(0) \Phi_{\alpha \beta}\left(\Phi_{\alpha \alpha}+\Phi_{\beta \beta}-2\right).$$
\end{proposition}
\begin{proof}
    The proof follows immediately from the proof of \Cref{prop: relu ode feature map} as well as combining  \Cref{thm: neural cov ode} and \Cref{lem: psd of limit}.
\end{proof}

\paragraph{PSD-ness of the SDE solution}
As a final remark, because the MLP conjugate kernel is PSD by definition, and that the SDE kernel \eqref{eqn: nsde informal} has a continuous path, (local) Skorokhod convergence implies weak convergence pointwise (up to each local stopping time). 
As a result, both convergences in \Cref{thm: nsde relu} and \Cref{thm: nsde smooth} guarantee positive semidefinite-ness on the Neural SDE solution $\Phi_t\succeq 0$ for any $t\geq 0$ in the limit. See \Cref{lem: psd of limit}.
\subsection{Special case 1: deep linear networks}\label{sec: special 1 linear}
Let us justify \texttt{Takeaway 1.1}, the main result in \cite{hanin2023bayesian}, as the first special case.
An important aspect of \Cref{prop: nsde 1} is that the NSDE separates two
effects in the Gram matrix $\Phi_t$ dynamic: a multiplicative stochastic component already present in deep linear networks and a drift
term induced by shaped nonlinearity. Indeed, with a deep linear network (identity activation function),  taking $\phi_s(x) = x$ yields equivalently $c_\pm = 0$ in \Cref{thm: nsde relu}. In this case, the NSDE becomes
$$\d \Phi_t=  \Phi_t^{1/2}\d B_t\Phi^{1/2}, \quad t_{\operatorname{depth}} = \ell/N,\quad t_{\operatorname{terminal}} = L/N$$
which is diffusion-only. The terminal time $t_{\operatorname{terminal}}$ parameterizes the action at fixed $P$. As in the main text, for a large $P\to\infty$, a  analysis of the operator norm yields
$$\|\Phi_t^{1/2}G\Phi^{1/2}_t\|_{op}\asymp \sqrt{P}\|\Phi_t\|_{op},$$
for a symmetric iid Gaussian $G$. To balance the sides we will naturally make a time-change $\d\tau = P\d t$ such that
\begin{equation}\label{eqn: linear NSDE tau}
    \d \Phi_\tau=  \frac{1}{\sqrt{P}}\Phi_\tau^{1/2}\d B_\tau\Phi_\tau^{1/2},\quad \tau_{\operatorname{depth}} = P\ell/N,\quad  \tau \in [0, \tau_{\operatorname{terminal}}=LP/N]
\end{equation}
which recovers the $LP/N$ rate independently derived in different ways from \cite{hanin2023bayesian, hanin2024bayesian, li2025geometric}.
Let us now solve \eqref{eqn: bayes partition def} in the case of deep linear networks by \eqref{eqn: linear NSDE tau}. Simple algebra reveals that $s_\tau(v) \triangleq v^{\top}\Phi_\tau v\in\R$ follows the SDE (for a fixed $p\in\R^P$)
$$\d s_\tau(v) = \frac{1}{\sqrt{P}}v^{\top}\Phi_\tau^{1/2}\d B_\tau\Phi_\tau^{1/2} v =\sqrt{\frac{2}{{P}}}s_\tau(v)\d W_\tau, \quad s_0 = v^{\top}\Phi_0v$$
where $(W_\tau)$ is a standard  Brownian motion in $\R$. This means that $s_\tau(p)$ is a Geometric Brownian motion with a closed terminal expression
\begin{equation}
    s_{\tau }(v) =_d s_0(v) \exp \left(-\frac{\tau }{P}+\sqrt{\frac{2 \tau }{P}} G\right), \quad G\sim\N(0, 1)
\end{equation}
and thus the partition function in the propositional limit reads (in this subsection we will use $Z^{(\tau)}$ to denote linear networks for convenience)
$$Z_\beta^{(\tau )}\left({x}_{0}, {\kappa}\right)\propto\int_{\R^P}\mathbb{E}\left[
    \exp\l(-\frac{1}{2 \beta}\|p\|^2+i p^{\top}Y-\frac{1}{2}\exp \left(-\frac{\tau }{P}+\sqrt{\frac{2 \tau }{P}} G\right)v^{\top}\Phi_0 v\r)\right] \d p$$
where the expectation is taken over $G\sim\N(0, 1)$. This simple derivation justifies the first part of \texttt{Takeaway 1.1}. To take the perturbative limit $\partial_\tau Z^{(\tau)}$ without differentiating under the integral, consider applying Ito's lemma on $s_\tau = v^\top\Phi_\tau v\geq 0$:\footnote{We drop the distinction between $\frac{1}{P}$ versus $\frac{1}{P+1}$ here because technically the NSDE is now applied on the $(P+1)\times (P+1)$ Gram matrix on $\D\cup\{x_0\}$. The $P$ vs. $P+1$ distinction is absorbed in $O(P^{-1})$ factors that will be present in the partition function in any case.}
$$\d s_\tau = \sqrt{\frac{2}{P}}s_\tau d B_\tau,\quad \d e^{-s_\tau/2} = e^{-s_\tau/2}\cdot \frac{1}{4P}s_\tau^2\d\tau + (\cdot) \d W_\tau$$
and as a result  (see \Cref{lem: differentiation of expectation} for justification)
$$\l.\partial _\tau\E[e^{-s_\tau/2}]\r|_{\tau=0} = \frac{1}{4P}\exp \left(-\frac{s_0}{2}\right) s_0^2$$
and since $\exp(-x/2), x\geq 0$ is bounded below and above
$$\l.\partial _\tau Z_{\beta}^{(\tau)}(x_0, \kappa)\r|_{\tau=0}  =  \int_{\R^P} \exp \left(-\frac{1}{2\beta}\|p\|^2+iY^\top p -\frac{s_0}{2}\right) \cdot\frac{1}{4P}s_0^2\d p$$
This is a Gaussian integral because $s_0$ is a quadratic function of $p$, so $\partial _\tau Z$ can be seen as the expectation of $\frac{1}{4P}s_0^2$ up to a normalizing constant. In fact, the normalizing constant is exactly
$$\int_{\R^P} \exp \left(-\frac{1}{2\beta}\|p\|^2+iY^\top p -\frac{1}{2}s_0\right)=Z^{(\tau=0)}_\beta\left(x_0, \kappa\right)$$
by \eqref{eqn: integral bayes partition}. As a result, we have that:
$$\l. {\partial _\tau \log Z_{\beta}^{(\tau)}(x_0, \kappa)}  \r|_{\tau=0}\
=\l.\frac{1}{Z^{(0)}_\beta\left(x_0, \kappa\right)} \partial _\tau Z_{\beta}^{(\tau)}(x_0, \kappa)\r|_{\tau=0}  = \E_{p\sim\N(\cdot, \cdot )}\l[\frac{1}{4P}s_0^2\r]$$
where $p$ follows the (complex) Gaussian distribution
\begin{equation}\label{eqn: int Gaussian def}
    p\sim \mathcal{N}\left( (\Phi^{\beta}_{pp})^{-1} (i  Y- \Phi_{p\kappa} \kappa), \quad  (\Phi^{\beta}_{pp})^{-1}\triangleq \l(\Phi_{pp}+\frac{1}{\beta}\I\r)^{-1}\right)
\end{equation}
where $\Phi_0 = \l[\begin{array}{cc}
     \Phi_{pp}\in\R^{P\times P} &  \Phi_{p\kappa}\\
      \Phi_{\kappa p}& \Phi_{\kappa\kappa} \in\R
\end{array}\r]$ is the decomposition of input Gram matrix. This gives us a closed-form formula for the first-order perturbation of the linear network partition function without integrating. 

Let us now focus on the $\beta=\infty$ case and evaluate the expectation directly 
\begin{equation}\label{eqn: s poly of kappa}
    s_0=p^\top\Phi_{pp}p+2\kappa p^\top\Phi_{p\kappa}+\kappa^2\Phi_{\kappa\kappa}
\end{equation} is a quadratic polynomial of a Gaussian. When $\beta=\infty$, simple calculation yields that:
\begin{equation}\label{eqn: linear exp 1}
    \E_{p\sim\eqref{eqn: int Gaussian def}} \l[s_0^2\r]=\l[ 
    \l(-P+ Y^\top\Phi_{pp}^{-1}Y- {\kappa^2} \l(
    \Phi_{\kappa\kappa}-\Phi_{\kappa p}\Phi_{pp}^{-1}\Phi_{p\kappa}
    \r)\r)^2
    +2P-4Y^\top\Phi_{pp}^{-1}Y
    \r]
\end{equation}
which directly proves \texttt{Takeaway 1.1}.
\paragraph{Dropping $O(P^{-1})$ terms.}
Expanding \eqref{eqn: linear exp 1} and denoting $\nu_0 = P^{-1}Y^\top\Phi_{pp}^{-1}Y$, $\|x_0^\perp\|^2 = \Phi_{\kappa\kappa}-\Phi_{\kappa p}\Phi_{pp}^{-1}\Phi_{p\kappa}$ gives
$$\l.\partial _\tau \log Z_{\infty}^{(\tau)}(x_0, \kappa)\r|_{\tau=0} =\frac{1}{4}\l[P(1-\nu_0)^2+2(1-2\nu_0)+2(1-\nu_0)\|x_0^\perp\|^2\kappa^2+\frac{1}{P}\|x_0^\perp\|^4\kappa^4\r]$$
is a polynomial of $\kappa$. Keeping only the top order (in terms of $P$) in each coefficient yields
$$\l.\partial _\tau \log Z_{\infty}^{(\tau)}(x_0, \kappa)\r|_{\tau=0} =\frac{1+O(P^{-1})}{4}\l[P(1-\nu_0)^2+2(1-\nu_0)\|x_0^\perp\|^2\kappa^2\r].$$
As a sanity check, this scale is correct because $\partial_\tau\log Z_\infty(0) = \partial_\tau\log \P(\D)\in\Theta(P)$ and that $\partial_\tau\var(f_{\operatorname{post}}(x_0))\in\Theta(1)$. Therefore, taking $P\to\infty$ is well-defined.
\paragraph{Positive temperature $\beta^{-1}>0$.}
Finally, let us address the effect of temperature. Plugging in the completed Gaussian $p\sim \mathcal{N}\left( (\Phi^{\beta}_{pp})^{-1} (i  Y- \Phi_{p\kappa} \kappa),  (\Phi^{\beta}_{pp})^{-1}\triangleq \l(\Phi_{pp}+\frac{1}{\beta}\I\r)^{-1}\right)$ into the expectation $\frac{1}{4P}\E\l[(p^\top\Phi_{pp}p+2\kappa p^\top\Phi_{p\kappa}+\kappa^2\Phi_{\kappa\kappa})^2\r]$ yields the polynomial expansion
$$\frac{1}{4 P} \mathbb{E}\left[s_0^2\right]=\frac{C_0}{4 P}+\frac{C_1}{4 P} \kappa+\frac{C_2}{4 P} \kappa^2+\frac{C_3}{P} \kappa^3+\frac{C_4}{4 P} \kappa^4,$$
where denoting $\I_\beta = \Phi_{pp} (\Phi_{pp}^\beta)^{-1}\prec \I$, we specifically need the following coefficients
$$
\begin{aligned}
C_0&=\l[\left(\operatorname{Tr}(\I_\beta)-Y^\top (\Phi_{pp}^\beta)^{-1}\I_\beta Y\right)\r]^2+2 \operatorname{Tr}\left((\I_\beta)^2\right)-4 Y^{\top} (\Phi_{pp}^\beta)^{-1} (\I_\beta)^2 Y, \\
C_3&= i\beta^{-1}\l[\Phi_{\kappa\kappa}-\Phi_{\kappa p}(\Phi_{pp}^\beta)^{-1}\Phi_{p\kappa}-\beta^{-1}\Phi_{\kappa p}(\Phi_{pp}^\beta)^{-2}\Phi_{p\kappa}\r]\Phi_{\kappa p} (\Phi_{pp}^\beta)^{-2}Y\\
C_4&= \l[\Phi_{\kappa\kappa}-\Phi_{\kappa p}(\Phi_{pp}^\beta)^{-1}\Phi_{p\kappa}-\beta^{-1}\Phi_{\kappa p}(\Phi_{pp}^\beta)^{-2}\Phi_{p\kappa}\r]^2.
\end{aligned}
$$
Under our standard assumptions on $\Phi_0$, one has that $\Phi_{\kappa\kappa}=\|x_0\|^2\in\Theta(1), \|\Phi_{\kappa p}\|^2\in\Theta(1)$ so $C_4\in O(1)$. Furthermore, $\|Y\|\in O(\sqrt{P})$ so $C_3\in  O(\sqrt{P})$. In other words,
$$\frac{1}{4 P} \mathbb{E}\left[s_0^2\right]=\frac{C_0}{4 P}+\frac{C_1}{4 P} \kappa+\frac{C_2}{4 P} \kappa^2+O(P^{-1/2})$$
is \textit{still} (approximately) a quadratic polynomial of $\kappa$.

Another identity we will use later on Bayesian evidence ($\kappa=0$) at positive temperature $\beta<\infty$ by plugging in $s_0 = p^\top \Phi_{pp}p$ and \eqref{eqn: int Gaussian def} into $\E[s_0^2]$ is that
$$\frac{1}{4P}\E [s_0^2]=\frac{1}{4P}\l[\left(\operatorname{Tr}(\I_\beta)-Y^\top (\Phi_{pp}^\beta)^{-1}\I_\beta Y\right)\r]^2+\frac{1}{2P} \operatorname{Tr}\left((\I_\beta)^2\right)-\frac{1}{P} Y^{\top} (\Phi_{pp}^\beta)^{-1} (\I_\beta)^2 Y$$
Keeping only the leading orders of large $P$, we are left with
\begin{equation}\label{eqn: linear evidence temp}
    \l.\partial_\tau\log Z^{\operatorname{(\tau)}}_\beta(0)\r|_{\tau=0}=\frac{1}{4P}\left[\operatorname{Tr}(\Phi_{pp}(\Phi_{pp}^\beta)^{-1})-Y^\top (\Phi_{pp}^\beta)^{-1}\Phi_{pp}(\Phi_{pp}^\beta)^{-1}Y\right]^2+O(1)
\end{equation}
in which $\beta=\infty$ recovers Corollary 3.9 in \cite{hanin2023bayesian}.

\subsection{Special case 2: Neural ODE in the infinite-width limit}\label{sec: special 2 neural ode}
Let us now study the second special case, which involves the infinite-width limit and $ t_{\operatorname{terminal}}= L/N\to 0$.
At a first glance, \eqref{eqn: nsde informal} when $t_{\operatorname{terminal}}=0$ will only return $\Phi_{L}=\Phi_0$ which leads to kernel method with the trivial identity feature map. Let us copy the NSDE with $s=\sqrt{N/\alpha_{\operatorname{act}}}$ and a fixed $P$ here as
$$\d \Phi_t=\alpha_{\operatorname{act}}\cdot b(\Phi_t) \d t+\Phi_t^{1/2}\d B_t\Phi_t^{1/2}, \quad t\in [0, L/N].$$
With a direct time-change $\d c= \alpha_{\operatorname{act}}\d t$, we have that
\begin{equation}\label{eqn: node informal}
    \d \Phi_c=b(\Phi_c) \d c+\frac{1}{\sqrt{\alpha_{\operatorname{act}}}}\Phi_c^{1/2}\d B_c\Phi_c^{1/2},\quad c\in [0, c_{\operatorname{terminal}} =  \alpha_{\operatorname{act}}t=L/s^2].
\end{equation}
\Cref{eqn: node informal} post-time change now allows sending $t\to 0$ so long as $c_{\operatorname{terminal}}>0$, and we end up with the ODE $\d \Phi_c=b(\Phi_c) \d c$. In other words, when we take the shaping  $ s = (c_{\operatorname{terminal}} /L)^{-1/2}$, the infinite-width limit admits a unique ODE parameterized by $c$.

Let us now be more precise in the exact limit taken and MLP convergence to the ODE by the following theorem in \cite{li2024differentialequationscalinglimits} which covers the limit $L = N^{1-\epsilon}$. The strict $N\to\infty$ first, then $L\to\infty$ (effectively $p=0$) limiting case has been derived using independent methods in \cite{zhang2022deep,hanin2024bayesian}. The resulting ODE remains unchanged.
\begin{proposition}[Neural Covariance ODE, Proposition 3.4 in \cite{li2024differentialequationscalinglimits}]\label{thm: neural cov ode}
Fix any $p\in (0, 1/2)$ and $P>0$ and consider the limit where $L = N^{2p}\to\infty, s \in\Theta(L^{1/2})$. Conjugate kernel evolution in the shaped ReLU MLP (resp. shaped smooth MLP)
converges to the drift-only ODE $\d \Phi_c=b(\Phi_c) \d c, c\in [0, L/s^2]$ weakly with respect to the Skorokhod topology of $D_{\R_+, \mathcal{S}^P}$.
\end{proposition}

Because the ODE removes stochasticity in the diagonals, the convergence is thus global and not just local.
As justified in \Cref{prop: relu ode feature map} and \Cref{prop: smooth ode feature map}, in both smooth and ReLU cases, the Neural Covariance ODE yields a deterministic, $\D$-independent feature map. Combining the above arguments and \Cref{thm: neural cov ode} into, we immediately get the following.
\begin{theorem}[\texttt{Takeaway 1.2}: Kernel method in the Neural Covariance ODE]\label{thm: ode feature map}
Fix any $p\in (0, 1/2)$ and $P>0$ and consider  the limit where $L = N^{2p}\to\infty, s \in\Theta(L^{1/2})$.
Bayesian inference is equivalent to a kernel method with feature map $f_0(x; c=L/s^2)$ defined by the Neural ODE that only depends on $s$ and $\phi_{\operatorname{base}}$.
\end{theorem}
Finally, when $L/s^2 = c\in o(1)$, the neural ODE gives the identity feature map, and when $c \to\infty$, the kernel becomes the solution to the ODE at $c=\infty$. It is not hard to check that in both cases of \Cref{thm: nsde relu} and \Cref{thm: nsde smooth}, the drift-only ODE at infinite-time converges to a single fixed stationary solution (depending on $\Phi_0$). Therefore, $s \in\Theta(L^{1/2})$ is the unique non-degenerate scaling.

\section{Perturbative expansion with small depth-time }\label{sec: small depth-time expansion}
In this section, we justify the takeaways in \Cref{takeaway: perturb} for the NSDE partition function when perturbing around small $\tau = LP/N$. Let us first present a modified Dynkin's formula (the proof deferred to \Cref{sec: lemmas}). 
\begin{lemma}[Differentiation of expectation] \label{lem: differentiation of expectation}
Assume that the coefficients $b: \mathbb{R}^d \to \mathbb{R}^d$ and $\sigma: \mathbb{R}^d \to \mathbb{R}^{d \times m}$ are locally Lipschitz and continuous.
Let $x_0 \in \mathbb{R}^d$, and let $\left(X_t\right)_{0 \leq t<\tau}$ be the unique local strong solution to the stochastic differential equation:
$$
\mathrm{d} X_t=b\left(X_t\right) \mathrm{d} t+\sigma\left(X_t\right) \mathrm{d} W_t, \quad X_0=x_0
$$
where $\tau=\lim_R\tau_R$ is the explosion stopping time such that  $b(\Phi_t), \sigma(\Phi_t)$ are bounded and Lipschitz within $[0, \tau_R]$.
Then, for any function $f \in C_b^2\left(\mathbb{R}^d\right)$, we have:
$$
\lim _{t \to  0} \frac{\mathbb{E}\left[f\left(X_t\right) \one_{\left\{t<\tau\right\}}\right]-f\left(x_0\right)}{t}=b\left(x_0\right)^{\top} \nabla f\left(x_0\right)+\frac{1}{2} \operatorname{Tr}\left(\sigma\left(x_0\right) \sigma\left(x_0\right)^{\top} \nabla^2 f\left(x_0\right)\right)
$$
In particular, if $\P(\tau=\infty)=1$, then
$$\left.\frac{\d}{\d t}\right|_{t=0^{+}} \mathbb{E}\left[f\left(X_t\right)\right]=b(x_0)^{\top} \nabla f(x_0)+\frac{1}{2} \operatorname{Tr}\left(\sigma(x_0) \sigma(x_0)^{\top} \nabla^2 f(x_0)\right) .$$
\end{lemma}
\Cref{lem: differentiation of expectation} immediately allow us to show \texttt{Takeaway 2} by a simple Ito's Lemma for a finite $P$ via the following arguments. Recall that $s_t(v) = v^\top\Phi_t v$, the above lemma applied on $F(\Phi) = \expp{-v^\top \Phi v/2}$ suggests that
\begin{equation}\label{eqn: partial diff s_tau}
    {\partial_\tau} \mathbb{E}\left[\exp \left(-\frac{s_\tau}{2}\right)\right]=\exp \left(-\frac{s_0}{2}\right)\left(-\frac{1}{2} c_{\operatorname{act}}\cdot v^\top b(\Phi_0)v+\frac{1}{4P} s_0^2\right)
\end{equation}
applied on \eqref{eqn: s_tau}. Very similar to the linear case (\Cref{sec: special 1 linear}), one can proceed to integrate over $p$ via the equivalent Gaussian integral \eqref{eqn: int Gaussian def}. To justify  \texttt{Takeaway 2}, let us first explain why differentiation and integral can be swapped, or in other words, why
\begin{equation}\label{eqn: swap partial int}
    \partial_\tau \int_{\R^P} e^{iY^\top p-\frac{1}{2\beta}\|p\|^2}\mathbb{E}\left[\exp \left(-\frac{s_\tau}{2}\right)\right]\d p=\int_{\R^P} e^{iY^\top p-\frac{1}{2\beta}\|p\|^2}\partial_\tau \mathbb{E}\left[\exp \left(-\frac{s_\tau}{2}\right)\right]\d p.
\end{equation}
By dominated convergence, we only need an upper bound $G(p)$ such that:
$$\left|\partial_\tau \mathbb{E} \exp \left(-\frac{s_\tau }{2}\right)\right| \leq G(p) \quad \text { uniformly for small } \tau,\text{ and}\quad \int_{\mathbb{R}^P} e^{-\frac{1}{2 \beta}\|p\|^2} G(p) \d p<\infty.$$
Let $H_\tau(p)\triangleq \mathbb{E}\left[\exp \left(-\frac{1}{2}s_\tau(v)\right)\right]$.
Since $\Phi_\tau \succeq 0$, we have $s_\tau(v) \geq 0$. Applying Itô's lemma to
$F_v(\Phi)=\exp \left(-\frac{1}{2} v^{\top} \Phi v\right)$
gives
$$
\d F_v\left(\Phi_\tau\right)=F_v\left(\Phi_\tau\right)\left(-\frac{1}{2} c_{\mathrm{act}} v^{\top} b\left(\Phi_\tau\right) v+\frac{1}{4 P} s_\tau(v)^2\right) \d \tau-\sqrt{\frac{1}{2 P}} F_v\left(\Phi_\tau\right) s_\tau(v) \d W_\tau
$$
Letting the localization radius go to infinity using non-explosion (see \Cref{lem: differentiation of expectation}) we get
$$
\partial_\tau H_\tau(p)=\mathbb{E}\left[e^{-s_\tau / 2}\left(-\frac{1}{2} c_{\mathrm{act}} v^{\top} b\left(\Phi_\tau\right) v+\frac{1}{4 P} s_\tau^2\right)\right].
$$
We now bound this uniformly for $\tau \in\left[0, \tau_0\right]$, where $\tau_0>0$ is fixed and small. First,
$$
0 \leq e^{-s_\tau / 2} s_\tau^2 \leq \sup _{x \geq 0} x^2 e^{-x / 2}=16 e^{-2}
$$
Thus
$\frac{1}{4 P} \mathbb{E}\left[e^{-s_\tau / 2} s_\tau^2\right] \leq \frac{4}{P e^2}.$ It remains to control the drift term. For both the ReLU and smooth NSDE drifts, it is not hard to see that there is a constant $C_b$, depending only on $P$ and the activation constants, such that for every $\Phi \succeq 0$,
$$
\|b(\Phi)\|_{op} \leq C_b\left(1+(\operatorname{Tr}(\Phi))^2\right) .
$$
Since the diagonal processes have finite moments on bounded time intervals, each diagonal process is dominated in moments by a geometric Brownian motion. Let $M_{\tau_0}\triangleq\sup _{0 \leq \tau \leq \tau_0} \mathbb{E}\left[1+\left(\operatorname{Tr} \Phi_\tau\right)^2\right]$ then it is finite and that:
$$\mathbb{E}\left[e^{-s_\tau / 2}\left|v^{\top} b\left(\Phi_\tau\right) v\right|\right]  \leq\|v\|^2 \mathbb{E}\left\|b\left(\Phi_\tau\right)\right\|_{op}  \leq C_b M_{\tau_0}\|v\|^2 .$$
Since $v=(p, \kappa)$ and $\kappa$ is fixed, we have
$\|v\|^2 \leq C_\kappa\left(1+\|p\|^2\right).$
So
$$\left|\partial_\tau H_\tau(p)\right| \leq C_{\tau_0, \kappa}\left(1+\|p\|^2\right)\triangleq G(p) \text{ is a polynomial}$$
where $C_{\tau_0, \kappa}$ is independent with $p, \tau$. This justifies dominated convergence in \eqref{eqn: swap partial int} when $\beta<\infty$.
Combining \eqref{eqn: partial diff s_tau} and \eqref{eqn: swap partial int}, we get the below which justifies \texttt{Takeaway 2}.
\begin{theorem}[First-order expansion of log partition function]
The first-order small $\tau$ partition function of the Neural Covariance SDE \eqref{eqn: nsde tau} with shaping \eqref{eqn: shape def} can be written as:
\begin{equation}\label{eqn: decompose d log Z}
    \l.\partial_\tau\log Z^{(\tau)}_{\beta}\r|_{\tau=0}= c_{\operatorname{act}}\cdot \E_{p\sim\eqref{eqn: int Gaussian def}}\l[ 
-\frac{1}{2} v^\top b(\Phi_0)v\r]+\E_{p\sim\eqref{eqn: int Gaussian def}}\l[ \frac{1}{4P} s_0^2\r]
\end{equation}
at any $0<\beta<\infty$.
\end{theorem}
The second term in \eqref{eqn: decompose d log Z} is the linear term from \Cref{sec: special 1 linear} based on the fact that the Gaussian distribution integrated over $p$ does not change at $\tau=0$. To see what the first term represents more clearly, consider again the Neural ODE kernel from \Cref{sec: special 2 neural ode}. From the ODE $\d\Phi_c = b(\Phi_c) \d c$ we have that:
$$\l.\d \expp{-\frac{1}{2}v^\top \Phi_c v}\r|_{c=0} =-\frac{1}{2}\expp{-\frac{1}{2}s_0} v^\top b(\Phi_0)v$$
Now we can swap in differentiation placement for the exact reason as above to get:
$$\l.\frac{1}{Z_\beta^{\operatorname{kernel}}(\cdot; f_0(\cdot; 0)=\operatorname{id})}\partial_c Z_\beta^{\operatorname{kernel}}(\cdot; f_0(\cdot; c))\r|_{c=0} = \E_{\eqref{eqn: int Gaussian def}}\lrb{-\frac{1}{2}v^\top b(\Phi_0) v}$$
which is exactly the first term in \eqref{eqn: decompose d log Z}, in turn justifying \Cref{eqn: perturbative meta}.

\subsection{First-order evidence and interpretation}\label{sec: interpretations}
From \Cref{eqn: perturbative meta}, the perturbative evidence \eqref{eqn: perturbative evidence} in \texttt{Takeaway 2.1} follows easily by taking $\kappa=0$ (so that $Z(x_0, 0)=Z(0)$). The non-linear component is copied verbatim, and the linear evidence follows from \eqref{eqn: linear evidence temp} directly.
Having now the exact expression for
$$\partial_\tau Z(0) = \partial_\tau \P_{\operatorname{prior}}(\D)$$
we can interpret the expression as architecture preference for the data. Specifically, we analyze the question of 
$$\text {when does a larger  } \frac{LP}{N} \text {  increase Bayesian evidence (deeper is better) }?$$
We do so by evaluating (recall \eqref{eqn: linear evidence temp}) $$\partial_\tau \log Z(0)=\frac{1}{4P}\left[\operatorname{Tr}(\Phi(\Phi^\beta)^{-1})-Y^\top (\Phi^\beta)^{-1}\Phi(\Phi^\beta)^{-1}Y\right]^2+c_{\operatorname{act}}{\partial_c \log Z^{\operatorname{kernel}}_{\beta}(0; f_0(\cdot, c))} $$ at $\tau=0$. For the data-generating process, we assume that there exists a norm $\rho>0$ such that
$$\|\Phi_0 - \rho\I_P\|_{op}\in o(1)$$
that input data are close to orthogonal to each other and all have norm close to $\rho$. Furthermore, we assume that $|\one^\top Y| =|\sum_\D Y_\alpha|\in\Theta(\sqrt{P})$. In this case, the linear evidence evaluates to the simplified expression
$$\frac{1}{4P}\left[\operatorname{Tr}(\Phi(\Phi^\beta)^{-1})-Y^\top (\Phi^\beta)^{-1}\Phi(\Phi^\beta)^{-1}Y\right]^2=\frac{P}{4}\l(\frac{\rho\|Y\|^2}{P(\rho+\beta^{-1})^2}-\frac{\rho}{\rho+\beta^{-1}}\r)^2 $$
Now for the non-linear part we need
$$\left.{\partial_c \log Z^{\operatorname{kernel}}_{\beta}(0; f_0(\cdot, c))} \right|_{c=0} =  \E_{p\sim \eqref{eqn: int Gaussian def}}\lrb{-\frac{1}{2} p^\top b(\Phi_0) p}.$$
Below we evaluate the drift-induced evidence for each different $b$.
\paragraph{Smooth activation function} In the smooth case we have:
$$b_{\alpha \beta}(\Phi)=c_1\left[\Phi_{\alpha \alpha} \Phi_{\beta \beta}+\Phi_{\alpha \beta}(2 \Phi_{\alpha \beta}-3)\right]+c_2 \Phi_{\alpha \beta}\left(\Phi_{\alpha \alpha}+\Phi_{\beta \beta}-2\right),\quad \forall \alpha, \beta\in [P]$$
where $c_1 =\frac{1}{4} (\phi''(0))^2, c_2 = \frac{1}{2} \phi'''(0)$. The Gaussian expectation evaluates to:
$$
\begin{aligned}
    \E_{\eqref{eqn: int Gaussian def}}\l[-\frac{1}{2} p^\top b(\Phi) p\r]=&\frac{\rho\left(2\left(c_1+c_2\right) \rho-\left(3 c_1+2 c_2\right)\right)}{2 (\rho+\beta^{-1} )^2}\|Y\|^2\\&+\frac{c_1 \rho^2}{2 (\rho+\beta^{-1} )^2}\left(\one^{\top} Y\right)^2-P\frac{\left(3 c_1+2 c_2\right) \rho(\rho-1)}{2 (\rho+\beta^{-1})}
\end{aligned}$$
As a result, we have the following statement.
\begin{proposition}[Architectural takeaway for smooth activation]
    Under the $\|\Phi_0 -\rho \I\|_{{op}}=o(1)$ training data $X_\D$ assumption, deeper is better, equivalently $\partial_\tau\P(\D)\geq 0$, if and only if:
    \begin{align*}
        0\leq &\frac{1}{4}\l(\frac{\rho\|Y\|^2}{P(\rho+\beta^{-1})^2}-\frac{\rho}{\rho+\beta^{-1}}\r)^2 +c_{\operatorname{act}}\cdot \l\{\frac{\rho\left(2\left(c_1+c_2\right) \rho-\left(3 c_1+2 c_2\right)\right)}{2 P(\rho+\beta^{-1} )^2}\|Y\|^2\r.\\&\quad +\l.\frac{c_1 \rho^2}{2 P(\rho+\beta^{-1} )^2}\left(\one^{\top} Y\right)^2- \frac{\left(3 c_1+2 c_2\right) \rho(\rho-1)}{2 (\rho+\beta^{-1})}\r\}
    \end{align*}
where $c_1 =\frac{1}{4} (\phi''(0))^2, c_2 = \frac{1}{2} \phi'''(0)$.
\end{proposition}
\paragraph{ReLU activation function} Similar to the smooth case, in ReLU networks recall that
$$b\left(\Phi_t\right)=\left(c_{+}-c_{-}\right)^2D_t^{1 / 2} \rho\left(D_t^{-1 / 2} \Phi_t D_t^{-1 / 2}\right) D_t^{1 / 2}$$
for $\rho(x)\triangleq\frac{1}{2 \pi}\left(\sqrt{1-x^2}-x\arccos x\right)$.
The Gaussian expectation in this case evaluates to:
$$\E_{\eqref{eqn: int Gaussian def}}\l[-\frac{1}{2} p^\top b(\Phi_0) p\r]= \frac{\left(c_{+}-c_{-}\right)^2\rho}{4 \pi(\rho+\beta^{-1})^2}\left[\left(\one ^{\top} Y\right)^2-\|Y\|^2\right].$$
As a result, we have the following statement.
\begin{proposition}[Architectural takeaway for ReLU activation]
    Under the $\|\Phi_0 -\rho \I\|_{{op}}=o(1)$ training data $X_\D$ assumption, deeper is better, equivalently $\partial_\tau\P(\D)\geq 0$, if and only if:
    \begin{align*}
        0\leq &\frac{1}{4}\l(\frac{\rho\|Y\|^2}{P(\rho+\beta^{-1})^2}-\frac{\rho}{\rho+\beta^{-1}}\r)^2  +c_{\operatorname{act}}\cdot \frac{\left(c_{+}-c_{-}\right)^2\rho}{4P \pi(\rho+\beta^{-1})^2}\left[\left(\one ^{\top} Y\right)^2-\|Y\|^2\right]
    \end{align*}
\end{proposition}

\subsection{First-order predictive posterior}\label{sec: first-order predictive posterior}
Let us apply \eqref{eqn: perturbative meta} to the predictive posterior at test $x_0$ and justify \texttt{Takeaway 2.2}. The predictive posterior can be read off from the partition function directly as the characteristic function
$$\log\E_{\operatorname{post}}[\exp\{-i\kappa f(x_0; \Theta)\}]= \log Z_\beta(x_0, \kappa)-\log Z_\beta(x_0, 0).$$
Because the constant ($\kappa=0$) terms are subtracted from the posterior, we will only work with $\kappa$-dependence and use a constant $C$ to placehold any quantity irrelevant to $\kappa$.
Taking differentiation on both sides, we have thus:
$$\partial_\tau\log\E_{\operatorname{post}, \tau}[\exp\{-i\kappa f(x_0; \Theta)\}]=\partial_\tau\l(\log Z^{(\tau)}_\beta(x_0, \kappa)-\log Z^{(\tau)}_\beta(x_0, 0)\r).$$
At $\tau=0$ the init condition of the forward pass, $\log\E_{\operatorname{post}}$ is a quadratic polynomial of $\kappa$ (as one can tell from \Cref{eqn: kernel method}). Furthermore, the posterior is a Gaussian random variable if and only if $\log\E_{\operatorname{post}}$ is a quadratic polynomial of $\kappa$. 

Let us examine the $\kappa$-dependence of $\log Z_\beta(x_0, \kappa)$ in light of \eqref{eqn: decompose d log Z}. Note that $$\E_{p\sim\eqref{eqn: int Gaussian def}}\l[ 
-\frac{1}{2} v^\top b(\Phi_0)v\r],\quad  v = [p^\top , \kappa]^\top $$ is \textit{by definition} a quadratic of $\kappa$. From \Cref{sec: special 1 linear}, we also know that $\E_{p\sim\eqref{eqn: int Gaussian def}}\l[ \frac{1}{4P} s_0^2\r]$ is (approximately) a quadratic polynomial of $\kappa$ for  any $\beta$. This gives us the following result.
\begin{proposition}[First order predictive posterior Gaussianity]\label{prop: first order predictive posterior gaussianity}
    For fixed $\kappa$, the first order perturbation $\partial_\tau\log\E_{\operatorname{post}}[e^{-i\kappa f(x_0)}]$ is a quadratic polynomial of $\kappa$ up to $O(P^{-1/2})$ terms at any $\beta>0$. As a result, at the first order of small $LP/N$ the predictive posterior remains Gaussian in the $P\to\infty$ limit.
\end{proposition}
Let us carefully compute the exact Gaussian for the simplified case $\beta\to \infty$, in which the higher order terms $\kappa^3, \kappa^4$ from linear networks are smaller at $O(P^{-1})$. 
Recall that (writing $s_0 = v^\top \Phi_0 v$ as a function of $\kappa$ and following notations in \Cref{sec: special 1 linear}):
$$\frac{1}{4P} \E_{p\sim\eqref{eqn: int Gaussian def}}\l[ s_0(\kappa)^2\r]=\frac{1+O(P^{-1})}{2}(1-\nu_0)\|x_0^\perp\|^2\kappa^2+C$$
where $C$ is independent with $\kappa$. Hence
$$ \partial_\tau\log Z^{(\tau)}_{\infty}= c_{\operatorname{act}}\cdot \E_{p\sim\eqref{eqn: int Gaussian def}}\l[ 
-\frac{1}{2} v^\top b(\Phi_0)v\r]+\frac{1+O(P^{-1})}{2}(1-\nu_0)\|x_0^\perp\|^2\kappa^2+C$$
Let us now find the respective feature map $f^{(\tau)}(\cdot)$ such that the induced kernel method \eqref{eqn: kernel method} has the same predictive posterior equivalent to the above.
Recall that the Neural ODE (\Cref{sec: special 2 neural ode}) gives that for 
$$ \partial_\tau\log Z^{\operatorname{kernel}}_{\infty}(x_0, \kappa; f_0(\cdot;\tau))= \E_{p\sim\eqref{eqn: int Gaussian def}}\l[ 
-\frac{1}{2} v^\top b(\Phi_0)v\r]+C$$
so
\begin{equation}\label{eqn: ode kernel partition perturb}
    \partial_\tau\log Z^{\operatorname{kernel}}_{\infty}(x_0, \kappa; f_0(\cdot; c_{\operatorname{act}}\tau))= c_{\operatorname{act}}\cdot \E_{\eqref{eqn: int Gaussian def}}\l[ 
-\frac{1}{2} v^\top b(\Phi_0)v\r]+C.
\end{equation}
Denote the quadratic polynomial expansion of $i\kappa$:
$$\E_{\eqref{eqn: int Gaussian def}}\l[ 
-\frac{1}{2} v^\top b(\Phi_0)v\r] = C-i\kappa c_1-\frac{1}{2}\kappa^2c_2, \quad c_1, c_2\in\R.$$
so
$$\log \frac{Z^{\operatorname{kernel}}_{\infty}(\cdot ; f_0)}{Z^{\operatorname{kernel}}_{\infty}(0; f_0)}=-i\kappa \l(\Phi_{\kappa p}\Phi_{pp}^{-1}Y+ \tau\cdot c_{\operatorname{act}}c_1\r)-\frac{1}{2}\kappa^2\l(\|x_0^\perp\|^2+\tau\cdot c_{\operatorname{act}}c_2\r)+O(\tau^2+P^{-1}).$$
Note that \eqref{eqn: kernel method} asserts (when $\beta=\infty$) that for kernel method the $\kappa^2$ coefficient in $\log Z$ only depends on the orthogonal component of test $f^{(\tau)}(x_0)$ to $f^{(\tau)}(\D)$, whereas the mean is determined from the linear combination of training labels. Here we see that the mean of first-order predictive posterior remains unchanged from the neural ODE, and the variance gets changed.  As a result, it is natural to define the feature map:
$$f^{(\tau)}: x_0\to \sum_{x_\mu\in \D}a_\mu f_0(x_\mu; c_{\operatorname{act}}\tau)+\lambda f_0^\perp(x_0; c_{\operatorname{act}}\tau)$$
where $f_0(x_\alpha)=f_0(x_\alpha)^{\perp}+\sum_{\mu\in\D}a_\mu f_0(x_\mu)$ and $f_0(x_\alpha)^{\perp}$ is orthogonal to $\operatorname{span}_\D(f_0(\cdot))$ (see also \eqref{eqn: kernel mean and var}). This yields the effective 
$$\log\frac{Z^{\operatorname{kernel}}(\cdot ; f^{(\tau)})}{Z^{\operatorname{kernel}}(0 ; f^{(\tau)})}=-i\kappa \l(\Phi_{\kappa p}\Phi_{pp}^{-1}Y+ \tau c_{\operatorname{act}}c_1\r)-\frac{1}{2}\kappa^2\lambda^2\l(\|x_0^\perp\|^2+\tau\cdot c_{\operatorname{act}}c_2\r)+O(\tau^2+P^{-1})$$
Note that \eqref{eqn: perturbative meta} evaluates $\log \frac{Z_{\infty}^{(\tau)}(x_0, \kappa )}{Z_{\infty}^{(\tau)}(0)}$ to
\begin{align*}
    -i\kappa \l(\Phi_{\kappa p}\Phi_{pp}^{-1}Y+ \tau c_{\operatorname{act}}c_1\r)-\frac{1}{2}\kappa^2\l(\|x_0^\perp\|^2+\tau\l[c_{\operatorname{act}}c_2+(\nu_0-1)\|x_0^\perp\|^2\r]\r)+O(\tau^2+P^{-1})
\end{align*}
we only need (\textit{independent} of $c_1$ and $c_2$)
$$\lambda^2=\frac{\|x_0^\perp\|^2+\tau\l[c_{\operatorname{act}}c_2+(\nu_0-1)\|x_0^\perp\|^2\r]}{\|x_0^\perp\|^2+\tau\cdot c_{\operatorname{act}}c_2}=1+(\nu_0-1)\tau+O(\tau^2)$$
for 
$$\log\frac{Z^{\operatorname{kernel}}(x_0, \kappa ; f^{(\tau)})}{Z^{\operatorname{kernel}}(0 ; f^{(\tau)})} = \log \frac{Z_{\infty}^{(\tau)}(x_0, \kappa )}{Z_{\infty}^{(\tau)}(0)}+O(\tau^2+P^{-1}).$$
To summarize the above, we have the following result.
\begin{theorem}[Equivalent feature map]\label{thm: first order feature map}
At zero temperature $\beta\to\infty$, kernel method induced by the feature map
$$f^{(\tau)}: x_0\to \sum_{x_\mu\in \D}a_\mu f_0(x_\mu; c_{\operatorname{act}}\tau)+\lambda f_0^\perp(x_0; c_{\operatorname{act}}\tau),\quad \lambda=1+\frac{1}{2}(\nu_0-1)\tau,$$
where $f_0(x_0; c_{\operatorname{act}}\tau = L/s^2)$ is the ODE feature map from \Cref{thm: ode feature map} and coefficients $a_\mu$ follow \eqref{eqn: kernel mean and var},
gives an equivalent predictive posterior characteristic function $\log Z(x_0, \kappa)$ (at any test point $x_0$) for Bayesian inference with NSDE up to $O(\tau^2+P^{-1})$ terms.
\end{theorem}

\section{MLP convergence to SDE}\label{sec: convergence apdx}
Finally, let us take care of the exact convergence from MLP to NSDE regarding the partition function.
We will show two statements: a pointwise convergence statement (\Cref{thm: mlp part to nsde part}) as well as a (stronger) statement on the perturbative  limit (\texttt{Takeaway 4}, see \Cref{prop: unif convergence} below). Because one cannot define the partial differentiation with discrete Markov Chains in the MLP (as opposed to the NSDE in \Cref{sec: small depth-time expansion}), our perturbative analysis for the MLP partition function will be translated through the \emph{uniform} convergence of layer-wise generators.
Denote:
$$F_v(\Phi) \triangleq \exp\l(-\frac{1}{2}v^\top\Phi v\r)\quad \text{ for some fixed }v.$$
We will use $\E^{\Phi_0}$ to denote the prior since both the MLP and SDE conjugate kernel dynamics are Markovian and depend only on $\Phi_0$. Note that $F$ is bounded by $(0, 1]$ by PSD-ness of both MLP and SDE kernel trajectories (see \Cref{sec: prior with nsde}).
In this section we will also use $\Phi_N^{(\ell)}$ to denote the conjugate kernel at layer $\ell$ for a width-$N$ MLP and use $\Phi_t$ to (only) denote the NSDE from \eqref{eqn: nsde informal} kernel without the $t\to\tau$ time change. 
Our layer-wise convergence result (formalizing \texttt{Takeaway 4}) can be stated as follows.
\begin{theorem}[Layer-wise uniform convergence]\label{prop: unif convergence} Fix a radius $r>0$ and a terminal time $T>0$. Let $K_r=\{\Phi: f_{\operatorname{test}}(\Phi) \leq r\}\cap\mathcal{S}_{+}^P$ for the $f_{\operatorname{test}}$ in \Cref{thm: nsde relu} and \Cref{thm: nsde smooth}. Define the stopping times
$$
\kappa_r^{(N)}\triangleq \frac{1}{N}\inf \left\{\ell \geq 0: \Phi_{N}^{(\ell)} \notin K_r\right\} , \quad \kappa_r\triangleq\inf \left\{t \geq 0: \Phi_t\notin K_r\right\},
$$
and the stopped difference quotient
$$
D_N^r F\left(t, \Phi_0\right)\triangleq N\left(\mathbb{E}^{\Phi_0}\left[F\left( \Phi_{N}^{(\lfloor t N\rfloor+1) \wedge N \kappa_r^{(N)}}\right)\right]-\mathbb{E}^{\Phi_0}\left[F\left( \Phi_{N}^{(\lfloor t N\rfloor \wedge N \kappa_r^{(N)})}\right)\right]\right),
$$
with limit
$$
D^r F\left(t, \Phi_0\right)\triangleq\E^{\Phi_0}\left[h\left(\Phi^{t \land \kappa_r}\right) \one_{\left\{t<\kappa_r\right\}}\right] .
$$
where $h\triangleq \mathcal{L}F$ is the generator.
Then in both cases of ReLU and smooth activations,
$$
\lim_{N\to\infty}\;\;\sup _{t \in[0, T], \Phi_0 \in K_{r/2}}\left|D_N^r F\left(t, \Phi_0\right)-D^r F\left(t, \Phi_0\right)\right| =0 .
$$
\end{theorem}
\begin{proof}[Proof of \Cref{prop: unif convergence}]
Because the MLP is a time-homogeneous Markov chain, let
$$
A_NF(\Phi)\triangleq
N\mathbb E^\Phi\!\left[
F(\Phi_N^{(1)})-F(\Phi)
\right],\quad D_N^r F\left(t, \Phi_0\right)=\mathbb{E}^{\Phi_0}\left[A_N F\left(\Phi_N^{\left(\lfloor tN \rfloor\right)}\right) \one_{\left\{\lfloor tN \rfloor<\kappa_r^{(N)}\right\}}\right]
$$
be the one-step generator of the unstopped discrete chain.
By definition we have \(F\in C_b^\infty\) (with constants depending only on $v$), and $h\triangleq \L F$ is continuous and bounded on \(K_r\).
\cite{li2022neural} shows convergence of the generator
$$\Phi_N^{(\ell+1)}-\Phi_N^{(\ell)}=\frac{1}{N}b_N\left(\Phi_N^{(\ell)}\right)+\frac{1}{\sqrt{N}}\sigma_N (\Phi_N^{(\ell)} ) \xi_{\ell+1}+O\left(N^{-3 / 2}\right)$$
with $b_N \to b, \sigma_N \sigma_N^{\top} \to \Sigma$ uniformly on compact sets and Lipschitz on $K_r$. From \cite[Lemma A.5 and Proposition A.6]{li2022neural}, we  have the local uniform generator convergence
$$
\varepsilon_N
\triangleq
\sup_{\Phi\in K_{r}}
\left|
A_NF(\Phi)-\mathcal L F(\Phi)
\right|
\to 0.
$$
Define the killed semigroups
$$
Q_{N, t}^r h\left(\Phi_0\right)\triangleq\mathbb{E}^{\Phi_0}\left[h\left(\Phi_N^{(\lfloor t N\rfloor)}\right) 
\one_{\left\{t<\kappa_r^{(N)}\right\}}\right],
\quad 
Q_t^r h\left(\Phi_0\right)\triangleq\mathbb{E}^{\Phi_0}\left[h\left(\Phi_t\right) \one_{\left\{t<\kappa_r\right\}}\right] .
$$
Then
$$
D_N^r F\left(t, \Phi_0\right)=Q_{N, t}^r\left(A_N F\right)\left(\Phi_0\right),
\quad
D^r F\left(t, \Phi_0\right)=Q_t^r(\mathcal{L} F)\left(\Phi_0\right) .
$$
Applying the Feller semigroup convergence criterion \cite[Theorem A.3]{li2022neural}, which gives convergence uniformly for bounded times once the generators converge on a core. Combined with \cite[Proposition A.6]{li2022neural} we get
$$
\eta_N
\triangleq
\sup_{t\in[0,T],\ \Phi_0\in K_
{r/2}}
\left|
Q_{N,t}^r h(\Phi_0)-Q_t^r h(\Phi_0)
\right|
\to 0.
$$
Although the core in \cite{li2022neural} is \(C_0^\infty\), our application is local (equivalently,
one may multiply \(F\) by a smooth cutoff equal to \(1\) on a neighborhood of \(K_{r}\)).
Combining the above and using that \(Q_{N,t}^r\) is a contraction in
the sup norm, we obtain
$$
\begin{aligned}
&\sup_{t\in[0,T],\ \Phi_0\in K_{r/2}} \left| D_N^r F(t,\Phi_0)-D^r F(t,\Phi_0)\right| 
\\\leq & 
\sup_{\Phi\in K_r}\left|A_NF(\Phi)-\mathcal L F(\Phi) \right|+
\sup_{t\in[0,T],\ \Phi_0\in K_{r/2}}\left| Q_{N,t}^r h(\Phi_0)-Q_t^r h(\Phi_0) \right| \\
= & \varepsilon_N+\eta_N \to 0
\end{aligned}
$$
which proves the claim in both the ReLU and smooth activation regimes.
\end{proof}
Let us conclude with the pointwise convergence in \Cref{thm: mlp part to nsde part}, in which we wanted to show that
\begin{align*}
\lim_{N\to\infty, L = t^\star N} Z_\beta^{(N, L, \D)} ({x}_{0}, {\kappa})\triangleq&\lim_{} \int_{\R^P}\exp\l[-\frac{1}{2 \beta}\|p\|^2+i p^{\top}Y\r] \mathbb{E}\left[ F(\Phi_N^{(L)})\right] \d p
    \\= &\int_{\R^P}\exp\l[-\frac{1}{2 \beta}\|p\|^2+i p^{\top}Y\r] \mathbb{E}_{\eqref{eqn: nsde informal}}\left[ F(\Phi_{t^\star})\right] \d p\triangleq Z_\beta^{(\tau)} ({x}_{0}, {\kappa} )
\end{align*}
\begin{proof}[Proof of \Cref{thm: mlp part to nsde part}]
    We start by showing that for any fixed $t>0$ and any $\Phi_0 \in \mathcal{S}_{+}^P$, one has that:
    \begin{equation}\label{eqn: F ptwise converge}
        \lim_{N\to\infty}
        \mathbb{E} \left[F\left(\Phi_{N}^{(\lfloor t N\rfloor)}\right)\right] = \mathbb{E} \left[F\left(\Phi_t\right)\right] .
    \end{equation}
for both ReLU and smooth activations (satisfying the assumptions {\texttt{A1-3}}). Once this is shown, then \Cref{thm: mlp part to nsde part} is concluded automatically by dominated convergence because $F$ is bounded.

For any large stopping radius $R$, again define
$\kappa_R^{(N)}$ and $\kappa_R$ are the local stopping times without finite-time explosion. The limiting process has continuous sample paths, and hence by the continuous mapping theorem for any fixed $t$:
$$
\Phi_N^{(\lfloor (t\land \kappa_R^{(N)})N\rfloor)}
\Rightarrow
\Phi_{t\land \kappa_R},  \quad \forall t
$$
by further the boundedness of $F$:
$$
\mathbb E^{\Phi_0}F\l(\Phi_N^{(\lfloor (t\land \kappa_R^{(N)})N\rfloor)}\r)
\to
\mathbb E^{\Phi_0}F\l(\Phi_{t\land \kappa_R}\r),\quad \forall t.
$$
We compare stopped and unstopped expectations. By boundedness
$$
\left|
\mathbb E^{\Phi_0}F\l(\Phi_N^{(\lfloor tN\rfloor)}\r)-\mathbb E^{\Phi_0}F\l(\Phi_N^{(\lfloor (t\land \kappa_R^{(N)})N\rfloor)}\r)
\right|
\leq
\P\l(\kappa_R^{(N)}\leq t\r).
$$
Similarly,
$$
\left|
\mathbb E^{\Phi_0}F\l(\Phi_{t}\r)-\mathbb E^{\Phi_0}F\l(\Phi_{t\land \kappa_R}\r)
\right|
\leq
\P(\kappa_R\le t).
$$
So we have
$$
\left|
\mathbb{E}^{\Phi_0} F\left(\Phi_N^{(\lfloor t N\rfloor)}\right) -\mathbb{E}^{\Phi_0} F\left(\Phi_t\right) 
\right|
\leq
\P(\kappa_R\le t)+\P\l(\kappa_R^{(N)}\leq t\r).
$$
By \Cref{lem: finite non-explosion}, sending $R\to\infty$ allows $\P(\kappa_R\le t)+\P\l(\kappa_R^{(N)}\leq t\r)\to 0$
and we are done.
\end{proof}
\section{Technical Lemmas and their proofs}\label{sec: lemmas}
\input{Lemmas}
\end{document}

%% file: pckgs.tex
\usepackage{amsmath,amssymb}
\usepackage[bookmarks, colorlinks=true, plainpages = false, citecolor = blue,linkcolor=blue,urlcolor = blue, filecolor = blue]{hyperref}
\usepackage{hyperref}
\usepackage{comment}
\usepackage{amsthm}
\usepackage{cleveref}
\usepackage{dsfont}

\renewcommand{\l}{\ell}
\newcommand{\E}{\mathbb{E}}
\newcommand{\R}{\mathbb{R}}
\renewcommand{\P}{\mathbb{P}}
\newcommand{\N}{\mathcal{N}}

\newcommand{\one}{\mathds{1}}

\usepackage{graphicx,graphics}
\usepackage{enumerate}

\newcommand{\var}{\textup{var}}
\newcommand{\cov}{\textup{cov}}

\newtheorem{lemma}{Lemma}

\newtheorem{theorem}{Theorem}

\newtheorem{proposition}{Proposition}
\newcommand{\D}{\mathcal{D}}

\newtheorem{definition}{Definition}

\renewcommand{\L}{\mathcal{L}}

\newcommand{\I}{\mathbb{I}}

\renewcommand{\tilde}{\widetilde}
\renewcommand{\l}{\left}
\renewcommand{\r}{\right}

\usepackage{algorithm}
\usepackage{algorithmic}
\usepackage{bm}
\newcommand{\lrp}[1]{\left( #1 \right)}
\newcommand{\lrb}[1]{\left[ #1 \right]}

\renewcommand{\d}{\mathrm{~d}}

\newcommand{\expp}[1]{\exp\lrp{#1}}

%% file: Lemmas.tex
\begin{proof}[Proof of \Cref{lem: finite non-explosion}]
Denote $F_r\triangleq\{x: f(x) \geq r\}$.
Let 
$Y_s^{n, r}\triangleq X_{s \wedge \tau_r^{(n)}}^n$ and $Y_s^r \triangleq X_{s \wedge \tau_r}.$
By local convergence,
$Y^{n, r} \Rightarrow Y^r$
in the Skorokhod topology.
Fix $r, t>0$. Choose $a>t$ such that the evaluation map $y \mapsto y(a)$ is a.s. continuous under $Y^r$. Such $a$ 's are dense because Càdlàg paths have at most countably many fixed-time discontinuities with positive probability.
Since $F_r$ is closed and the paths are right-continuous, whenever $\tau_r^{(n)}<\infty$, the hitting point satisfies
$$
X_{\tau_r^{(n)}}^n \in F_r
$$
Hence, if $\tau_r^{(n)} \leq t<a$, then the stopped process has already hit $F_r$, so
$$
Y_a^{n, r}=X_{\tau_r^{(n)}}^n \in F_r,\quad 
\left\{\tau_r^{(n)} \leq t\right\} \subseteq\left\{Y_a^{n, r} \in F_r\right\}.
$$
Thus
$$
\limsup _n \mathbb{P}\left(\tau_r^{(n)} \leq t\right) \leq \limsup _n \mathbb{P}\left(Y_a^{n, r} \in F_r\right)
$$
By the continuous mapping theorem, $Y_a^{n, r} \Rightarrow Y_a^r$. Since $F_r$ is closed, Portmanteau gives
$$
\limsup _n \mathbb{P}\left(Y_a^{n, r} \in F_r\right) \leq \mathbb{P}\left(Y_a^r \in F_r\right)
$$
But $Y_a^r \in F_r$ if and only if $\tau_r \leq a$. Hence
$$
\limsup _n \mathbb{P}\left(\tau_r^{(n)} \leq t\right) \leq \mathbb{P}\left(\tau_r \leq a\right)\quad \forall a>t.
$$
Now take $a \downarrow t$ and by continuity of probability from above, we get the desired claim.
\end{proof}

\begin{proof}[Proof of \Cref{lem: differentiation of expectation}]
Denote $Lf(x)\triangleq b(x)^\top \nabla f(x)
+\frac12 \operatorname{Tr}\!\left(\sigma(x)\sigma(x)^\top \nabla^2 f(x)\right).$
By Itô's formula applied to \(X_{t\wedge \tau_R}\),
\[
f(X_{t\wedge \tau_R})-f(x_0)
=
\int_0^t \mathbf 1_{\{s<\tau_R\}}Lf(X_s)\,ds
+
\int_0^t \mathbf 1_{\{s<\tau_R\}}
\nabla f(X_s)^\top \sigma(X_s)\,dW_s .
\]
The stochastic integral is a martingale, hence
\[
\frac{\mathbb E[f(X_{t\wedge \tau_R})]-f(x_0)}{t}
=
\frac1t\int_0^t
\mathbb E\!\left[\mathbf 1_{\{s<\tau_R\}}Lf(X_s)\right]\,ds .
\]
As $s\to 0$, \(X_s\to x_0\) a.s. and
\(\mathbf 1_{\{s<\tau_R\}}\to1\) a.s. Therefore, by dominated convergence,
\[
\mathbb E\!\left[\mathbf 1_{\{s<\tau_R\}}Lf(X_s)\right]
\longrightarrow Lf(x_0),
\]
and consequently
\[
\lim_{t\to 0}
\frac{\mathbb E[f(X_{t\wedge \tau_R})]-f(x_0)}{t}
=Lf(x_0).
\]
It remains to compare this with the killed process. Since \(f\) is bounded,
\[
\left|
\mathbb E[f(X_t)\mathbf 1_{\{t<\tau \}}]
-
\mathbb E[f(X_{t\wedge \tau_R})]
\right|
\le
2\|f\|_\infty \mathbb P(\tau_R\le t).
\]
We show that \(\mathbb P(\tau_R\le t)=o(t)\) for small $t$ for any $R$ such that $\tau_{R/2}>0$ a.s. Let $K_R \triangleq  \{\Phi:  \max\{(\min_\alpha \Phi_{\alpha\alpha})^{-1}, \max_\alpha\Phi_{\alpha\alpha}\}\leq R\}\cap \mathcal{S}_{+}$ be a compact domain. Let  $K_b\triangleq \sup_{K_R}\|b\|_{\infty}$  and $K_\sigma\triangleq \sup_{K_R}\|\sigma(x)\|_\infty$ and they are both finite.
For \(u\le t\),
\[
X_{u\wedge \tau_R}-x_0
=
\int_0^{u\wedge \tau_R}b(X_s)\,ds
+
\int_0^{u\wedge \tau_R}\sigma(X_s)\,dW_s .
\]
If \(t\) is small enough that \(\|X_0\|_{\infty}+K_b t\le R/2\), then
\[
\{\max_\alpha(\Phi_{t\land \tau_R})_{\alpha\alpha}=R\}
\subseteq
\left\{
\sup_{u\in [0, t]}
\left|
\int_0^{u\wedge \tau_R}\sigma(X_s)\,dW_s
\right|
\ge R/2
\right\}.
\]
By the Burkholder--Davis--Gundy inequality,
$$
\P\l(\max_\alpha(\Phi_{t\land \tau_R})_{\alpha\alpha}=R\r)
\le
\frac{C}{R^4}
\mathbb E\left[
\sup_{u\in [0, t]}
\left|
\int_0^{u\wedge \tau_R}\sigma(X_s)\,dW_s
\right|^4
\right]
\le
\frac{C K_\sigma^4}{R^4}t^2
=o(t).
$$
Similarly, $\P\l(\max_\alpha(\Phi_{t\land \tau_R})^{-1}_{\alpha\alpha}=R\r)\in o(t)$ as well.
Hence
\[
\frac{
\mathbb E[f(X_t)\mathbf 1_{\{t<\tau \}}]
-
\mathbb E[f(X_{t\wedge \tau_R})]
}{t}
\to 0.
\]
Combining the preceding limits yields exactly
$
\lim_{t\downarrow0}
\frac{\mathbb E[f(X_t)\mathbf 1_{\{t<\tau \}}]-f(x_0)}{t}
=
Lf(x_0).
$
\end{proof}

\begin{lemma}[Equivalent Neural Covariance SDE in matrix form]\label{lem: NSDE matrix}
Given a symmetric positive-definite $\Phi\in\R^{m\times m}$, let $M = m(m+1) / 2$ and define
$\Sigma\triangleq \left[\Phi_{\alpha \gamma} \Phi_{\beta \delta}+\Phi_{\alpha \delta} \Phi_{\beta \gamma}\right]_{\alpha \leq \beta, \gamma \leq \delta} \in\R^{M\times M}.$ Suppose $B\in\R^M$ has i.i.d. $\N(0, 1)$ entries and $G\in\R^{m\times m}$ is symmetric and has $\N(0, 1)$ off-diagonal and $\N(0, 2)$ diagonal entries, then
$$\l(\Sigma^{1/2}B\r)_{\alpha\leq \beta}=_d\l(\Phi^{1/2}G\Phi^{1/2}\r)_{\alpha\leq \beta}$$
equal in distribution for all $\alpha, \beta\in [m]$.
\end{lemma}
\begin{proof}
    Let
$
A\triangleq \Phi^{1 / 2}=A^{\top},
$
then we only need to show that the covariance matrix of $(A G A)_{\alpha\leq \beta}\in\R^{M}$ is exactly $\Sigma$. For any $\alpha, \beta, \gamma, \delta$,
$$
(A G A)_{\alpha \beta}=\sum_{i, j=1}^m A_{\alpha i} G_{i j} A_{\beta j} .
$$
Since  $\mathbb{E}\left[G_{i j} G_{k \ell}\right]=\delta_{i k} \delta_{j \ell}+\delta_{i \ell} \delta_{j k}$ for all $i, j, k, \ell$:
$$
\begin{aligned}
\mathbb{E}\left[(A G A)_{\alpha \beta}(A G A)_{\gamma \delta}\right] & =\sum_{i, j, k, \ell} A_{\alpha i} A_{\beta j} A_{\gamma k} A_{\delta \ell} \mathbb{E}\left[G_{i j} G_{k \ell}\right] \\
& =\sum_{i, j, k, \ell} A_{\alpha i} A_{\beta j} A_{\gamma k} A_{\delta \ell}\left(\delta_{i k} \delta_{j \ell}+\delta_{i \ell} \delta_{j k}\right) \\
& =\sum_{i, j} A_{\alpha i} A_{\beta j} A_{\gamma i} A_{\delta j}+\sum_{i, j} A_{\alpha i} A_{\beta j} A_{\gamma j} A_{\delta i}
\\ & =\left(\sum_i A_{\alpha i} A_{\gamma i}\right)\left(\sum_j A_{\beta j} A_{\delta j}\right)+\left(\sum_i A_{\alpha i} A_{\delta i}\right)\left(\sum_j A_{\beta j} A_{\gamma j}\right) \\
& =\left(A A^{\top}\right)_{\alpha \gamma}\left(A A^{\top}\right)_{\beta \delta}+\left(A A^{\top}\right)_{\alpha \delta}\left(A A^{\top}\right)_{\beta \gamma}
\end{aligned}
$$
Because $A=\Phi^{1 / 2}$ is symmetric, $A A^{\top}=\Phi$. Therefore
$$
\mathbb{E}\left[(A G A)_{\alpha \beta}(A G A)_{\gamma \delta}\right]=\Phi_{\alpha \gamma} \Phi_{\beta \delta}+\Phi_{\alpha \delta} \Phi_{\beta \gamma} .
$$
Since the mean of $G$ is zero, this concludes the proof. This proof also justifies the difference in form between our \eqref{eqn: nsde informal} versus the vector version in \cite{li2022neural}.
\end{proof}
\begin{lemma}[Bayesian partition function to conjugate kernel]\label{lem: Bayes partition}
    For any model $f(x) = W_{\operatorname{out}}^{\top}h(x)$  with a linear layer as its last where $W_{\operatorname{out}}\sim\N(0, \I_N)$, the Bayes partition function of the predictive posterior \eqref{eqn: bayes partition def} under MSE loss at test point $x_0$ (and training data $X\in\R^{\cdot \times P}, Y\in\R^P$):
    $$Z_\beta\left({x}_{0}, {\kappa}\right)= (2 \pi \beta)^{P / 2} \mathbb{E}_{\operatorname{prior}}\left[\exp \left[-\frac{\beta}{2}\left\|Y-W_{\operatorname{out}}^{\top}h(X)\right\|_2^2-i\kappa f({x}_{0})\right]\right]$$
    is equal to
    $$Z_\beta\left({x}_{0}, {\kappa}\right)= \int_{\R^P}\exp\l[-\frac{1}{2 \beta}\|p\|^2+i p^{\top}Y\r]\cdot \mathbb{E}_{\operatorname{prior}}\left[
    \exp\l(-\frac{v^{\top}\Phi v}{2}\r)\right] \d p$$
    where $v = [p^{\top}, \kappa]^{\top}\in\R^{P+1}$ and $\Phi = \l[h(X), h(x_0)\r]^{\top}\l[h(X), h(x_0)\r]\in\R^{(P+1)\times (P+1)}$.
\end{lemma}
\begin{proof}
From the Gaussian integral
$$
1=\int_{\mathbb{R}^P} \frac{\d t}{(2 \pi \beta)^{P / 2}} \exp \left[-\frac{1}{2 \beta}\left\|t^{\top}-i \beta\left(Y-W_{\operatorname{out}}^{\top}h(X)\right)\right\|^2\right]
$$
we have
$$
\exp \left[-\frac{\beta}{2}\left\|Y-W_{\operatorname{out}}^{\top}h(X)\right\|^2\right] =\int_{\R^P} \frac{\d t}{(2 \pi \beta)^{P / 2}} \exp \left[-\frac{1}{2 \beta}\|t\|^2+i t^{\top}\left(Y-W_{\operatorname{out}}^{\top}h(X)\right)\right]
$$
so
$$
\begin{aligned}
    Z_\beta\left({x}_{0}, {\kappa}\right)&= (2 \pi \beta)^{P / 2} \mathbb{E}_{\operatorname{prior}}\left[\exp \left[-\frac{\beta}{2}\left\|Y-W_{\operatorname{out}}^{\top}h(X)\right\|_2^2-i\kappa f({x}_{0})\right]\right]\\
    &= \mathbb{E}_{\operatorname{prior}}\left[\int_{\R^P}\exp \left[-\frac{1}{2 \beta}\|t\|^2+i t^{\top}Y-it^{\top}W_{\operatorname{out}}^{\top}h(X)-i\kappa W_{\operatorname{out}}^{\top}h({x}_{0})\right]\d t\right]\\
    &= \int_{\R^P}\exp\l[-\frac{1}{2 \beta}\|t\|^2+i t^{\top}Y\r]\cdot \mathbb{E}_{\operatorname{prior}}\left[
    \exp\l(-\frac{[t^{\top}, \kappa]\Phi [t^{\top}, \kappa]^{\top}}{2}\r)\right] \d t
\end{aligned}
$$
where in the last line, we integrate over $W_{\operatorname{out}}$ to get the desired result.
\end{proof}
\begin{lemma}[PSD of the limit]\label{lem: psd of limit}
Suppose $X^n \Rightarrow X$ locally in $D_{\mathbb{R}_{+}, \mathbb{R}^{m \times m}}$ under test function $f$ without finite-time explosion. If the pre-limit is PSD
$$
\mathbb{P}\left(X_t^n \in \mathcal{S}_{+}, \forall t \geq 0\right)=1 \quad \text { for every } n,
$$
then the limit is also PSD:
$
\mathbb{P}\left(X_t \in \mathcal{S}_{+}, \forall t \geq 0\right)=1.
$
\end{lemma}
\begin{proof}
Let
$
D\left(\mathcal{S}_{+}\right)\triangleq\left\{x \in D_{\mathbb{R}_{+}, \mathbb{R}^{m \times m}}: x(t) \in \mathcal{S}_{+} \forall t \geq 0\right\}.$
First note that $D\left(\mathcal{S}_{+}\right)$is closed in the Skorokhod topology. Indeed, $\mathcal{S}_{+} \subseteq \mathbb{R}^{m \times m}$ is closed. If $x^k \rightarrow x$ in Skorokhod topology and $x^k(t) \in \mathcal{S}_{+}$for all $t$, then, on every compact interval, there exist time changes $\lambda_k$ such that
$$
\sup _{s \leq T}\left\|x^k\left(\lambda_k(s)\right)-x(s)\right\| \rightarrow 0 .
$$
Since $\lambda_k$ is onto, $x^k\left(\lambda_k(s)\right) \in \mathcal{S}_{+}$. Thus $x(s)$ is a limit of points in $\mathcal{S}_{+}$, so $x(s) \in \mathcal{S}_{+}$. Hence $x \in D\left(\mathcal{S}_{+}\right)$.
Fix $r>0$, and define the stopped processes
$$Y_t^{n, r}:=X_{t \wedge \tau_r^{(n)}}^n, \quad Y_t^r:=X_{t \wedge \tau_r}.$$
By local convergence,
$
Y^{n, r} \Rightarrow Y^r
$
in the Skorokhod topology. Since $X^n$ lives in $\mathcal{S}_{+}$a.s., so does $Y^{n, r}$; hence
$$
\mathbb{P}\left(Y^{n, r} \in D\left(\mathcal{S}_{+}\right)\right)=1 .
$$
Because $D\left(\mathcal{S}_{+}\right)$ is closed, Portmanteau gives
$$
1=\lim \sup \mathbb{P}\left(Y^{n, r} \in D\left(\mathcal{S}_{+}\right)\right) \leq \mathbb{P}\left(Y^r \in D\left(\mathcal{S}_{+}\right)\right) .
$$
so $\mathbb{P}\left(Y^r \in D\left(\mathcal{S}_{+}\right)\right)=1$.
Now use non-explosion. On the event $\left\{\tau_r=\infty\right\}$, we have
$$
Y_t^r=X_{t \wedge \tau_r}=X_t,\quad 
\left\{\tau_r=\infty\right\} \cap\left\{Y^r \in D\left(\mathcal{S}_{+}\right)\right\} \subseteq\left\{X \in D\left(\mathcal{S}_{+}\right)\right\} .
$$
Since $\mathbb{P}\left(Y^r \in D\left(\mathcal{S}_{+}\right)\right)=1$, it follows that
$$
\mathbb{P}\left(X \in D\left(\mathcal{S}_{+}\right)\right) \geq \mathbb{P}\left(\tau_r=\infty\right) .
$$
Using non-explosion,
$
\lim _{r \rightarrow \infty} \mathbb{P}\left(\tau_r=\infty\right)=1,
$
we obtain
$
\mathbb{P}\left(X_t \in \mathcal{S}_{+} \forall t \geq 0\right)=1 .
$
\end{proof}

%% file: ref.bib
@article{li2022neural,
  title={The neural covariance SDE: Shaped infinite depth-and-width networks at initialization},
  author={Li, Mufan and Nica, Mihai and Roy, Dan},
  journal={Advances in Neural Information Processing Systems},
  volume={35},
  pages={10795--10808},
  year={2022}
}

@article{hanin2019finite,
  title={Finite depth and width corrections to the neural tangent kernel},
  author={Hanin, Boris and Nica, Mihai},
  journal={arXiv preprint arXiv:1909.05989},
  year={2019}
}

@misc{hanin2025globaluniversalitysingularvalues,
      title={Global Universality of Singular Values in Products of Many Large Random Matrices}, 
      author={Boris Hanin and Tianze Jiang},
      year={2025},
      eprint={2503.07872},
      archivePrefix={arXiv},
      primaryClass={math.PR},
      url={https://arxiv.org/abs/2503.07872}, 
}

@article{hanin2024bayesian,
  title={Bayesian inference with deep weakly nonlinear networks},
  author={Hanin, Boris and Zlokapa, Alexander},
  journal={arXiv preprint arXiv:2405.16630},
  year={2024}
}

@article{hanin2023bayesian,
  title={Bayesian interpolation with deep linear networks},
  author={Hanin, Boris and Zlokapa, Alexander},
  journal={Proceedings of the National Academy of Sciences},
  volume={120},
  number={23},
  pages={e2301345120},
  year={2023},
  publisher={National Academy of Sciences}
}

@inproceedings{hayou2022curse,
  title = 	 {The Curse of Depth in Kernel Regime},
  author =       {Hayou, Soufiane and Doucet, Arnaud and Rousseau, Judith},
  booktitle = 	 {Proceedings on "I (Still) Can't Believe It's Not Better!" at NeurIPS 2021 Workshops},
  pages = 	 {41--47},
  year = 	 {2022},
  publisher =    {PMLR},
  url = 	 {https://proceedings.mlr.press/v163/hayou22a.html},
}

@article{chizat2019lazy,
  title={On lazy training in differentiable programming},
  author={Chizat, Lenaic and Oyallon, Edouard and Bach, Francis},
  journal={Advances in neural information processing systems},
  volume={32},
  year={2019}
}

@misc{jacot2020neuraltangentkernelconvergence,
      title={Neural Tangent Kernel: Convergence and Generalization in Neural Networks}, 
      author={Arthur Jacot and Franck Gabriel and Clément Hongler},
      year={2020},
      eprint={1806.07572},
      archivePrefix={arXiv},
      primaryClass={cs.LG},
      url={https://arxiv.org/abs/1806.07572}, 
}

@article{matthews2018gaussian,
  title={Gaussian process behaviour in wide deep neural networks},
  author={Matthews, Alexander G de G and Rowland, Mark and Hron, Jiri and Turner, Richard E and Ghahramani, Zoubin},
  journal={arXiv preprint arXiv:1804.11271},
  year={2018}
}

@article{lee2017deep,
  title={Deep neural networks as gaussian processes},
  author={Lee, Jaehoon and Bahri, Yasaman and Novak, Roman and Schoenholz, Samuel S and Pennington, Jeffrey and Sohl-Dickstein, Jascha},
  journal={arXiv preprint arXiv:1711.00165},
  year={2017}
}

@article{hayou2019training,
  title={Training dynamics of deep networks using stochastic gradient descent via neural tangent kernel},
  author={Hayou, Soufiane and Doucet, Arnaud and Rousseau, Judith},
  journal={CoRR},
  year={2019}
}

@misc{du2019gradientdescentprovablyoptimizes,
      title={Gradient Descent Provably Optimizes Over-parameterized Neural Networks}, 
      author={Simon S. Du and Xiyu Zhai and Barnabas Poczos and Aarti Singh},
      year={2019},
      eprint={1810.02054},
      archivePrefix={arXiv},
      primaryClass={cs.LG},
      url={https://arxiv.org/abs/1810.02054}, 
}

@misc{yang2021tensorprogramsiiineural,
      title={Tensor Programs III: Neural Matrix Laws}, 
      author={Greg Yang},
      year={2021},
      eprint={2009.10685},
      archivePrefix={arXiv},
      primaryClass={cs.NE},
      url={https://arxiv.org/abs/2009.10685}, 
}

@article{hanin24random,
  author  = {Boris Hanin},
  title   = {Random Fully Connected Neural Networks as Perturbatively Solvable Hierarchies},
  journal = {Journal of Machine Learning Research},
  year    = {2024},
  volume  = {25},
  number  = {267},
  pages   = {1--58},
  url     = {http://jmlr.org/papers/v25/23-0643.html}
}

@misc{noci2023shapedtransformerattentionmodels,
      title={The Shaped Transformer: Attention Models in the Infinite Depth-and-Width Limit}, 
      author={Lorenzo Noci and Chuning Li and Mufan Bill Li and Bobby He and Thomas Hofmann and Chris Maddison and Daniel M. Roy},
      year={2023},
      eprint={2306.17759},
      archivePrefix={arXiv},
      primaryClass={stat.ML},
      url={https://arxiv.org/abs/2306.17759}, 
}

@article{Hanin2019products,
   title={Products of Many Large Random Matrices and Gradients in Deep Neural Networks},
   volume={376},
   ISSN={1432-0916},
   url={http://dx.doi.org/10.1007/s00220-019-03624-z},
   DOI={10.1007/s00220-019-03624-z},
   number={1},
   journal={Communications in Mathematical Physics},
   publisher={Springer Science and Business Media LLC},
   author={Hanin, Boris and Nica, Mihai},
   year={2019},
   month=dec, pages={287–322} }

@article{martens2021rapid,
  title={Rapid training of deep neural networks without skip connections or normalization layers using deep kernel shaping},
  author={Martens, James and Ballard, Andy and Desjardins, Guillaume and Swirszcz, Grzegorz and Dalibard, Valentin and Sohl-Dickstein, Jascha and Schoenholz, Samuel S},
  journal={arXiv preprint arXiv:2110.01765},
  year={2021}
}

@misc{li2025geometric,
      title={Geometric Dyson Brownian Motions and the Free Log-Normal Limit for a Non-Square Product of Random Matrices}, 
      author={Mufan Li and Jaume de Dios Pont and Mihai Nica and Daniel M. Roy},
      year={2026},
}

@misc{trevisan2023widedeepneuralnetworks,
      title={Wide Deep Neural Networks with Gaussian Weights are Very Close to Gaussian Processes}, 
      author={Dario Trevisan},
      year={2023},
      eprint={2312.11737},
      archivePrefix={arXiv},
      primaryClass={math.ST},
      url={https://arxiv.org/abs/2312.11737}, 
}

@misc{bassetti2024proportionalinfinitewidthinfinitedepthlimit,
      title={Proportional infinite-width infinite-depth limit for deep linear neural networks}, 
      author={Federico Bassetti and Lucia Ladelli and Pietro Rotondo},
      year={2024},
      eprint={2411.15267},
      archivePrefix={arXiv},
      primaryClass={stat.ML},
      url={https://arxiv.org/abs/2411.15267}, 
}

@article{polson2013bayesian,
  title={Bayesian inference for logistic models using P{\'o}lya--Gamma latent variables},
  author={Polson, Nicholas G and Scott, James G and Windle, Jesse},
  journal={Journal of the American statistical Association},
  volume={108},
  number={504},
  pages={1339--1349},
  year={2013},
  publisher={Taylor \& Francis}
}

@article{hanin2021non,
  title={Non-asymptotic results for singular values of Gaussian matrix products},
  author={Hanin, Boris and Paouris, Grigoris},
  journal={Geometric and Functional Analysis},
  volume={31},
  number={2},
  pages={268--324},
  year={2021},
  publisher={Springer}
}

@misc{wang2024nonlinearspikedcovariancematrices,
      title={Nonlinear spiked covariance matrices and signal propagation in deep neural networks}, 
      author={Zhichao Wang and Denny Wu and Zhou Fan},
      year={2024},
      eprint={2402.10127},
      archivePrefix={arXiv},
      primaryClass={stat.ML},
      url={https://arxiv.org/abs/2402.10127}, 
}

@article{cho2009kernel,
  title={Kernel methods for deep learning},
  author={Cho, Youngmin and Saul, Lawrence},
  journal={Advances in neural information processing systems},
  volume={22},
  year={2009}
}

@inproceedings{Fan20Spectra,
 author = {Fan, Zhou and Wang, Zhichao},
 booktitle = {Advances in Neural Information Processing Systems},
 editor = {H. Larochelle and M. Ranzato and R. Hadsell and M.F. Balcan and H. Lin},
 pages = {7710--7721},
 publisher = {Curran Associates, Inc.},
 title = {Spectra of the Conjugate Kernel and Neural Tangent Kernel for linear-width neural networks},
 url = {https://proceedings.neurips.cc/paper_files/paper/2020/file/572201a4497b0b9f02d4f279b09ec30d-Paper.pdf},
 volume = {33},
 year = {2020}
}

@article{el2010spectrum,
  title={The spectrum of kernel random matrices},
  author={El Karoui, Noureddine},
  year={2010}
}

@article{chouard2023deterministic,
  title={Deterministic equivalent of the conjugate kernel matrix associated to artificial neural networks},
  author={Chouard, Cl{\'e}ment},
  journal={arXiv preprint arXiv:2306.05850},
  year={2023}
}

@misc{benigni2022largesteigenvaluesconjugatekernel,
      title={Largest Eigenvalues of the Conjugate Kernel of Single-Layered Neural Networks}, 
      author={Lucas Benigni and Sandrine Péché},
      year={2022},
      eprint={2201.04753},
      archivePrefix={arXiv},
      primaryClass={math.PR},
      url={https://arxiv.org/abs/2201.04753}, 
}

@article{pennington2017nonlinear,
  title={Nonlinear random matrix theory for deep learning},
  author={Pennington, Jeffrey and Worah, Pratik},
  journal={Advances in neural information processing systems},
  volume={30},
  year={2017}
}

@phdthesis{Nea94,
  author  = {Neal, Radford M.},
  title   = {Bayesian Learning for Neural Networks},
  school  = {Department of Computer Science, University of Toronto},
  year    = {1994},
  url     = {https://glizen.com/radfordneal/ftp/thesis.pdf},
  note    = {Ph.D. thesis},
}

@inproceedings{Wil96,
 author = {Williams, Christopher},
 booktitle = {Advances in Neural Information Processing Systems},
 editor = {M.C. Mozer and M. Jordan and T. Petsche},
 pages = {},
 publisher = {MIT Press},
 title = {Computing with Infinite Networks},
 url = {https://proceedings.neurips.cc/paper_files/paper/1996/file/ae5e3ce40e0404a45ecacaaf05e5f735-Paper.pdf},
 volume = {9},
 year = {1996}
}

@article{SCS88,
  title = {Chaos in Random Neural Networks},
  author = {Sompolinsky, H. and Crisanti, A. and Sommers, H. J.},
  journal = {Phys. Rev. Lett.},
  volume = {61},
  issue = {3},
  pages = {259--262},
  numpages = {0},
  year = {1988},
  month = {7},
  publisher = {American Physical Society},
  doi = {10.1103/PhysRevLett.61.259},
  url = {https://link.aps.org/doi/10.1103/PhysRevLett.61.259}
}

@misc{Poo+16,
      title={Exponential expressivity in deep neural networks through transient chaos}, 
      author={Ben Poole and Subhaneil Lahiri and Maithra Raghu and Jascha Sohl-Dickstein and Surya Ganguli},
      year={2016},
      eprint={1606.05340},
      archivePrefix={arXiv},
      primaryClass={stat.ML},
      url={https://arxiv.org/abs/1606.05340}, 
}

@misc{Sch+17,
      title={Deep Information Propagation}, 
      author={Samuel S. Schoenholz and Justin Gilmer and Surya Ganguli and Jascha Sohl-Dickstein},
      year={2017},
      eprint={1611.01232},
      archivePrefix={arXiv},
      primaryClass={stat.ML},
      url={https://arxiv.org/abs/1611.01232}, 
}

@article{Mac92,
  author    = {MacKay, David J. C.},
  title     = {A Practical {B}ayesian Framework for Backpropagation Networks},
  journal   = {Neural Computation},
  volume    = {4},
  number    = {3},
  pages     = {448--472},
  year      = {1992},
  doi       = {10.1162/neco.1992.4.3.448},
}

@misc{Blu+15,
      title={Weight Uncertainty in Neural Networks}, 
      author={Charles Blundell and Julien Cornebise and Koray Kavukcuoglu and Daan Wierstra},
      year={2015},
      eprint={1505.05424},
      archivePrefix={arXiv},
      primaryClass={stat.ML},
      url={https://arxiv.org/abs/1505.05424}, 
}

@misc{Gal+16,
      title={Dropout as a Bayesian Approximation: Representing Model Uncertainty in Deep Learning}, 
      author={Yarin Gal and Zoubin Ghahramani},
      year={2016},
      eprint={1506.02142},
      archivePrefix={arXiv},
      primaryClass={stat.ML},
      url={https://arxiv.org/abs/1506.02142}, 
}

@misc{li2024differentialequationscalinglimits,
      title={Differential Equation Scaling Limits of Shaped and Unshaped Neural Networks}, 
      author={Mufan Bill Li and Mihai Nica},
      year={2024},
      eprint={2310.12079},
      archivePrefix={arXiv},
      primaryClass={stat.ML},
      url={https://arxiv.org/abs/2310.12079}, 
}

@article{hanin2019universal,
  title={Universal function approximation by deep neural nets with bounded width and relu activations},
  author={Hanin, Boris},
  journal={Mathematics},
  volume={7},
  number={10},
  pages={992},
  year={2019},
  publisher={MDPI}
}

@misc{daubechies2022nonlinear,
      title={Nonlinear Approximation and (Deep) ReLU Networks}, 
      author={I. Daubechies and R. DeVore and S. Foucart and B. Hanin and G. Petrova},
      year={2019},
      eprint={1905.02199},
      archivePrefix={arXiv},
      primaryClass={cs.LG},
      url={https://arxiv.org/abs/1905.02199}, 
}

@article{zhang2022deep,
  title={Deep learning without shortcuts: Shaping the kernel with tailored rectifiers},
  author={Zhang, Guodong and Botev, Aleksandar and Martens, James},
  journal={arXiv preprint arXiv:2203.08120},
  year={2022}
}

@misc{bassetti2025featurelearningfinitewidthbayesian,
      title={Feature learning in finite-width Bayesian deep linear networks with multiple outputs and convolutional layers}, 
      author={Federico Bassetti and Marco Gherardi and Alessandro Ingrosso and Mauro Pastore and Pietro Rotondo},
      year={2025},
      eprint={2406.03260},
      archivePrefix={arXiv},
      primaryClass={stat.ML},
      url={https://arxiv.org/abs/2406.03260}, 
}

@misc{zhang2017understandingdeeplearningrequires,
      title={Understanding deep learning requires rethinking generalization}, 
      author={Chiyuan Zhang and Samy Bengio and Moritz Hardt and Benjamin Recht and Oriol Vinyals},
      year={2017},
      eprint={1611.03530},
      archivePrefix={arXiv},
      primaryClass={cs.LG},
      url={https://arxiv.org/abs/1611.03530}, 
}

@article{Belkin19reconciling,
   title={Reconciling modern machine-learning practice and the classical bias–variance trade-off},
   volume={116},
   ISSN={1091-6490},
   url={http://dx.doi.org/10.1073/pnas.1903070116},
   DOI={10.1073/pnas.1903070116},
   number={32},
   journal={Proceedings of the National Academy of Sciences},
   publisher={Proceedings of the National Academy of Sciences},
   author={Belkin, Mikhail and Hsu, Daniel and Ma, Siyuan and Mandal, Soumik},
   year={2019},
   month=jul, pages={15849–15854} 
}

@misc{advani2017highdimensionaldynamicsgeneralizationerror,
      title={High-dimensional dynamics of generalization error in neural networks}, 
      author={Madhu S. Advani and Andrew M. Saxe},
      year={2017},
      eprint={1710.03667},
      archivePrefix={arXiv},
      primaryClass={stat.ML},
      url={https://arxiv.org/abs/1710.03667}, 
}

@article{Li_2021,
   title={Statistical Mechanics of Deep Linear Neural Networks: The Backpropagating Kernel Renormalization},
   volume={11},
   ISSN={2160-3308},
   url={http://dx.doi.org/10.1103/PhysRevX.11.031059},
   DOI={10.1103/physrevx.11.031059},
   number={3},
   journal={Physical Review X},
   publisher={American Physical Society (APS)},
   author={Li, Qianyi and Sompolinsky, Haim},
   year={2021},
   month=sep }

@article{Pacelli_2023,
   title={A statistical mechanics framework for Bayesian deep neural networks beyond the infinite-width limit},
   volume={5},
   ISSN={2522-5839},
   url={http://dx.doi.org/10.1038/s42256-023-00767-6},
   DOI={10.1038/s42256-023-00767-6},
   number={12},
   journal={Nature Machine Intelligence},
   publisher={Springer Science and Business Media LLC},
   author={Pacelli, R. and Ariosto, S. and Pastore, M. and Ginelli, F. and Gherardi, M. and Rotondo, P.},
   year={2023},
   month=dec, pages={1497–1507} }

@misc{camilli2025informationtheoreticreductiondeepneural,
      title={Information-theoretic reduction of deep neural networks to linear models in the overparametrized proportional regime}, 
      author={Francesco Camilli and Daria Tieplova and Eleonora Bergamin and Jean Barbier},
      year={2025},
      eprint={2505.03577},
      archivePrefix={arXiv},
      primaryClass={math.ST},
      url={https://arxiv.org/abs/2505.03577}, 
}

@misc{roberts2022principles,
  title={The principles of deep learning theory},
  author={Roberts, Daniel A and Yaida, Sho and Hanin, Boris},
  year={2022},
  publisher={Cambridge University Press Cambridge, MA, USA}
}

@article{rotskoff2018neural,
  title={Neural networks as interacting particle systems: Asymptotic convexity of the loss landscape and universal scaling of the approximation error},
  author={Rotskoff, Grant M and Vanden-Eijnden, Eric},
  journal={stat},
  volume={1050},
  pages={22},
  year={2018}
}

@article{mei2018mean,
  title={A mean field view of the landscape of two-layer neural networks},
  author={Mei, Song and Montanari, Andrea and Nguyen, Phan-Minh},
  journal={Proceedings of the National Academy of Sciences},
  volume={115},
  number={33},
  pages={E7665--E7671},
  year={2018},
  publisher={National Academy of Sciences}
}

@article{sirignano2020mean,
  title={Mean field analysis of neural networks: A law of large numbers},
  author={Sirignano, Justin and Spiliopoulos, Konstantinos},
  journal={SIAM Journal on Applied Mathematics},
  volume={80},
  number={2},
  pages={725--752},
  year={2020},
  publisher={SIAM}
}

@article{seroussi2023separation,
  title={Separation of scales and a thermodynamic description of feature learning in some cnns},
  author={Seroussi, Inbar and Naveh, Gadi and Ringel, Zohar},
  journal={Nature Communications},
  volume={14},
  number={1},
  pages={908},
  year={2023},
  publisher={Nature Publishing Group UK London}
}

@article{naveh2021predicting,
  title={Predicting the outputs of finite deep neural networks trained with noisy gradients},
  author={Naveh, Gadi and Ben David, Oded and Sompolinsky, Haim and Ringel, Zohar},
  journal={Physical Review E},
  volume={104},
  number={6},
  pages={064301},
  year={2021},
  publisher={APS}
}
